\documentclass[12pt,leqno]{article}

\usepackage{amsfonts,amsthm,amssymb,amsmath,graphicx}
\usepackage{color}

\newcommand{\fb}{\mathfrak {b}}

\newcommand{\fg}{\mathfrak {g}}

\newcommand{\fh}{\mathfrak {h}}
\newcommand{\fn}{\mathfrak {n}}
\newcommand{\ft}{\mathfrak {t}}
\newcommand{\fa}{\mathfrak {a}}
\newcommand{\fk}{\mathfrak {k}}

\newcommand{\fl}{\mathfrak {l}}
\newcommand{\fp}{\mathfrak {p}}
\newcommand{\fq}{\mathfrak {q}}

\newcommand{\fo}{\mathfrak {o}}
\newcommand{\fs}{\mathfrak {s}}
\newcommand{\fu}{\mathfrak {u}}

\newcommand{\gz}{\mathfrak {g}_{\bar{0}}}

\newcommand{\ch}{\mathcal H}
\newcommand{\Lie}{\mbox{\rm Lie}}

\newcommand{\si}{\sigma}

\newcommand{\sk}{\bigskip}

\newcommand{\lra}{\longrightarrow}

\newcommand{\pa}{\partial}

\newcommand{\De}{\Delta}
\newcommand{\Dez}{\Delta_{\bar{0}}}

\newcommand{\al}{\alpha}
\newcommand{\be}{\beta}
\newcommand{\de}{\delta}
\newcommand{\ep}{\varepsilon}

\newcommand{\la}{\lambda}
\newcommand{\na}{\nabla}

\newcommand{\om}{\omega}

\newcommand{\iimath}{\jmath}

\newcommand{\sh}{\sharp}

\newcommand{\bsl}{\backslash}

\newcommand{\im}{{\rm Im}}

\newcommand{\ad}{{\rm ad}}

\newtheorem{proposition}{Proposition}[section]
\newtheorem{theorem}[proposition]{Theorem}

\newtheorem{definition}[proposition]{Definition}
\newtheorem{example}[proposition]{Example}
\newtheorem{remark}[proposition]{Remark}

\newcommand{\bl}{\mathbb {L}}

\newcommand{\bc}{\mathbb {C}}

\newcommand{\br}{\mathbb {R}}

\newcommand{\tiC}{C}

\newcommand{\lat}{\la}

\newcommand{\gss}{\fg_{\rm ss}^\si}
\newcommand{\fas}{\fa_\sigma}

\newcommand{\R}{{\mathbb {R}}}

\newcommand{\cH} {\mathcal{H}}

\newcommand{\cO} {\mathcal{O}}

\newcommand{\QQ} {\mathcal{Q}}
\newcommand{\RR} {\mathcal{R}}

\newcommand{\fso}{{\mathfrak s}{\mathfrak o}}
\newcommand{\fsp}{{\mathfrak s}{\mathfrak p}}

\newcommand{\Gz}{G_{\bar{0}}}
\newcommand{\omz}{\omega_{\bar{0}}}
\newcommand{\blz}{\bl_{\bar{0}}}

\newcommand{\Om}{\Omega}

\newcommand{\we}{\wedge}
\newcommand{\bp}{\bar{\partial}}
\newcommand{\ba}{{\bar{0}}}

\newcommand{\beq}{\begin{equation}}
\newcommand{\eeq}{\end{equation}}

\newcommand{\fag}{\fl}

\newcommand{\zero}{{\bar{0}}}
\newcommand{\one}{{\bar{1}}}
\newcommand{\Deo}{\Delta_{\bar{1}}}

\newcommand{\zR}{S}

\newcommand{\0}{{\bar{0}}}
\newcommand{\1}{{\bar{1}}}
\newcommand{\tf}{\widetilde{F}}

\newcommand{\xs}{(X_\si)_\0}

\begin{document}

{\large \bf Geometric quantization and
unitary

highest weight Harish-Chandra supermodules}

\bigskip

\bigskip

Meng-Kiat Chuah

Department of Mathematics

National Tsing Hua University

Hsinchu, Taiwan

{\tt chuah@math.nthu.edu.tw}

\bigskip

Rita Fioresi

FaBiT, University of Bologna

Bologna, Italy

{\tt rita.fioresi@unibo.it}

\sk

\bigskip

\noindent {\bf Abstract:}

Geometric quantization transforms a symplectic manifold with Lie group action
to a unitary representation.
In this article, we extend geometric quantization to the super setting.
We consider real forms of contragredient Lie supergroups with compact Cartan subgroups,
and study their actions on some pseudo-K\"ahler supermanifolds.
We construct their unitary
representations in terms of sections of some line bundles.
These unitary representations contain highest weight Harish-Chandra supermodules, whose
occurrences depend on the image of the moment map.
As a result, we construct a Gelfand model of highest weight Harish-Chandra supermodules.
We also perform symplectic reduction, and show that
quantization commutes with reduction.

\sk


\sk



\newpage
\section{Introduction}
\setcounter{equation}{0}
Let $G$ be a real semisimple Lie group.
Geometric quantization \cite{ko} associates a $G$-invariant
symplectic manifold $X$
to a unitary $G$-representation $\cH$,
where $\cH$ consists of some sections of line bundles on $X$.
For $G$ with compact Cartan subgroup,
we apply geometric quantization to obtain a family of
discrete series representations in \cite{jfa}.
In view of great mathematical interests in supersymmetry
(see for instance \cite{de}\cite{ma}\cite{vsv2}
and references therein),
it is natural to extend the above result to the super setting.
In this article, we consider Berezin integration \cite[\S4.6]{vsv2}
on invariant measure of $X$
to provide an $L^2$-structure on $\ch$.
This leads to a super
unitary $G$-representation \cite{fg}
on a family of highest weight Harish-Chandra supermodules \cite{cfv}.
We also consider the notions of
moment map \cite[\S26]{gs} and symplectic reduction \cite{gs},
and study their effects on
the unitary representation.
We now explain these projects in more details.

In what follows we use the word
``ordinary'' for the theory of Lie algebras, Lie groups and their
representations, to make a distinction from the theory of
their super counterparts.
We shall discuss
many concepts such as pseudo-K\"ahler, geometric quantization,
unitary representation,
moment map, symplectic reduction, and it is understood
that they are handled in the super setting.
To avoid lengthy terminology,
we refrain from adding the word ``super''
such as super pseudo-K\"ahler and so on.

Let $L$ be a complex Lie supergroup, with $\fl=\Lie(L)$
a contragredient Lie superalgebra.
So $\fl$ is one of $\fs \fl(m,n)$, $B(m,n)$, $C(n)$, $D(m,n)$,
$D(2,1;\al)$, $F(4)$, $G(3)$
\cite[\S2]{ka}.
Let $\fg$ be a real form of $\fl$,
namely $\fg$ is a real subalgebra such that $\fg \oplus i \fg = \fl$ \cite{pa}.
We assume that
$\fg$ has a compact Cartan subalgebra $\ft$.
Let $\fa = i\ft$, so that $\fh=\ft + \fa$
is a Cartan subalgebra of $\fl$.
It leads to a root space decomposition
$\fl = \fh + \sum_\De \fl_\al$, where $\fl_\al$ is the root space
of $\al \in \De$.
By a choice of positive system $\De^+$,
we let $\fn$ consist of the positive root spaces.
Then $\fl = \fg + \fa + \fn$ is the Iwasawa decomposition.
We always let the uppercase Roman letters be the subgroups
for the subalgebras in lowercase Gothic letters,
so for instance $\fa = \Lie(A)$ and so on.

We are interested in the supermanifold
\[ X = GA . \]
We express the real supergroups by
\[
\mbox{super Harish-Chandra pair (SHCP) } \;
 G=(\Gz, \fg) \; , \; X=(\Gz A, \fg+\fa) ,
\]
where $\Gz$ and $\Gz A$ are the ordinary Lie groups with Lie algebras
$\gz$ and $\gz +\fa$ respectively \cite[Prop.7.4.15]{ccf}.
We will treat $X$ as a supermanifold
with symplectic structures and supersymmetries.
By the Iwasawa decomposition,
$GAN$ is an open subset of $L$, so
we can identify
$X$ with an open subset of the complex superspace
$L/N$
(see \cite{cfv}).
Hence $X$ is a complex supermanifold with $G \times H$-action
(namely left $G$ and right $H$) because $H$ normalizes $N$.

For convenience, we make the convention that
$\fh^*$ denotes the real subspace (\ref{keci})
of the dual space of $\fh$, so that $\fh^* \cong i \ft^* \cong \fa^*$.
See Remark \ref{ide}.

We fix a nondegenerate invariant super symmetric
bilinear form $B$ on $\fg$. It extends to $\fl$ by $\bc$-linearity,
and is unique up to multiplication by non-zero scalar \cite[Prop.2.5.5]{ka}.
For each $\al \in \De^+$, we have the coroot
\begin{equation}
h_\al \in \fa \;,\; B(h_\al, v) = \al(v)
\label{coroot}
\end{equation}
for all $v \in \fh$.
We define the regular elements
\[ \fa_{\rm reg}^* = \{\la \in \fa^* \;;\;
\la(h_\al) \neq 0 \mbox{ for all } \al \in \De^+\}. \]
Then $\fa_{\rm reg}^*$ is a disjoint union of open cones, known as Weyl chambers.

We shall discuss in Section 3 the $G \times T$-invariant differential forms on $X$.
In particular, a symplectic form is a closed nondegenerate 2-form,
and a pseudo-K\"ahler form is a symplectic form of type $(1,1)$
with respect to the complex structure of $X$.
We shall compute the moment map \cite[\S26]{gs}
of the right $T$-action on the pseudo-K\"ahler form,
\beq
 \Phi : X \lra i \ft^* .
 \label{mdf}
 \eeq
The original formula of moment map has image in $\ft^*$,
but we intend to relate it to
the integral weights $\la \in i \ft^*$ of $T$-action,
so we add a factor $i$ and obtain (\ref{mdf}).

We identify a $G \times T$-invariant function
on $X$ with $F \in C^\infty(A)$,
namely $F : A \lra \br$ is a smooth function.
By the exponential map, we identify $F$ with
$\tf \in C^\infty(\fa)$ by $F(e^v) = \tf(v)$.
In Remark \ref{ide}, we use the derivative of $\tf$ to
define the gradient map $F' : A \lra \fa^*$.
Also, we say that $F$ is nondegenerate
(resp. strictly convex)
if the Hessian matrix of $\tf$ is
nondegenerate
(resp. positive definite)
everywhere.
The assertion $\Phi(ga) = \frac{1}{2}F'(a)$ below makes use of
$i \ft^* \cong \fa^*$.

\begin{theorem}
Every $G \times T$-invariant closed (1,1)-form on $X$ can be expressed as
$\om = i \pa \bp F$, where $F \in C^\infty(A)$.
It is pseudo-K\"ahler if and only if $F$ is
nondegenerate and $\mbox{\rm Im}(F') \subset \fa_{\rm reg}^*$.
The right $T$-action has moment map
$\Phi(ga) = \frac{1}{2}F'(a)$.
\label{thm1}
\end{theorem}

We shall use the pseudo-K\"ahler form to obtain
unitary realizations of an important class of $G$-representations,
known as Harish-Chandra highest weight supermodules $\Theta_{\la + \rho}$.
Let $\De_{\bar{0}}, \De_{\bar{1}}$ be the even and odd roots.
Let $\fk$ be the Lie subalgebra of $\fl_\0$ containing $\fh$, such that
$\fk \cap \fg_\0$ is a maximal compact subalgebra of $\fg_\0$.
We say that an even root is compact if it is a root of $\fk$,
and is noncompact otherwise.
We write $\De_{\bar{0}} = \De_c \cup \De_n$
for the compact and noncompact roots.
So for example $\De_c^+$ are the positive compact roots.
The coroots $h_\al$ of (\ref{coroot}) are used to define
\beq
 \widetilde{\tiC} = \{\la \in \fh^* \;;\;
\la(h_\al) \geq 0 \mbox{ for all } \al \in \De_c^+ \,,\,
\la(h_\be) < 0 \mbox{ for all } \be \in \De_n^+ \cup \De_{\bar{1}}^+ \} .
\label{hccone}
\eeq
We call $\widetilde{\tiC}$ the Harish-Chandra cone,
since for the ordinary setting, it is defined
in \cite[\S5]{hc-v} to study Harish-Chandra representations.
The irreducible highest weight supermodules $\Theta_{\la + \rho}$
are parametrized by integral weights $\la \in \fh^*$ in the set
\begin{equation}
C = \{\la \in \widetilde{\tiC}  \;;\;
(\la + \rho)(h_\al) < 0
\mbox{ for all } \al \in \Delta^+_n \cup \De_{\bar{1}}^+ \},
\label{vgood}
\end{equation}
where
\[ \rho = \frac{1}{2} (\sum_{\Dez^+} \al - \sum_{\Deo^+} \al) .\]
The condition in (\ref{vgood})
assures that $\Theta_{\la + \rho}$ is irreducible, and moreover
$\la$ is typical.
See also \cite[(1.7)]{jfa} for the ordinary setting.

In general $\tiC$ may be empty.
To avoid this, we recall the notion of admissible system.
Let $\fk = \fh + \sum_{\De_c} \fl_\al$ and
$\fq^\pm = \sum_{\De_n^\pm \cup \De_\1^\pm} \fl_\al$,
so that $\fl = \fk + \fq^+ + \fq^-$.
The positive system is called admissible if
\beq [\fk, \fq^\pm ] \subset \fq^\pm \;\;,\;\;
[\fq^\pm, \fq^\pm] = 0 .\label{howa}
\eeq
Then $\tiC \neq \emptyset$ if and only if the positive system
is admissible \cite[Thm.1.1]{cf}.
The real forms with admissible systems are
classified in \cite[Figs.1-6]{cf}.
We now assume that the positive system is admissible,
so that $\tiC \neq \emptyset$.

Fix a $G \times T$-invariant
pseudo-K\"ahler form $\om = i \pa \bp F$ on $X$, where $F$ is strictly convex.
We first construct a super line bundle $\bl$ on $X$.
In the ordinary setting,
the canonical line bundle $\bl_\ba$ over $X_\ba = G_\zero A$
has the property that the Chern class of $\bl_\ba$
is the cohomology class $[\om_\ba]$ \cite{ko}.
Here we extend $\bl_\ba$ to $\bl$, by fixing a global
splitting of the complex supergroup $X=GA$.

The line bundle $\bl_\ba$ has a connection $\na$
whose curvature is $\om_\ba$.
A section $s$ of $\bl_\ba$ is said to be
holomorphic if $\na_v s = 0$ for all anti-holomorphic vector fields $v$.
We extend the holomorphic structure to $\bl$, via the fixed global
splitting of $X$.
Let $\cH(\bl)$ denote the space of holomorphic sections on $\bl$.

For each integral weight $\la \in i \ft^*$, let
 $\cH(\bl)_{\lat}$ denote the holomorphic sections that
transform by $\lat$ under the right $T$-action (see (\ref{puku})).
Let $W(\bl)_{\lat}$ be the smallest SHCP representation of $(G_{\bar{0}},\fg)$
in $\cH(\bl)_{\lat}$
which contains the highest weight vector.
For $\la \in C$, as irreducible $G$-modules \cite[Thm.1]{cfv},
\[  \Theta_{\la + \rho} \cong W(\bl)_{\lat} .\]
Here we apply $\fh^* \cong i \ft^* \cong \fa^*$ as explained in
Remark \ref{ide}. By this isomorphism, we can regard $C$ as a subset of
$\fh^*$, $i \ft^*$ or $\fa^*$.
We assume that $\Theta_{\la + \rho}$ is unitarizable; see Proposition \ref{eeej}.

There is an invariant Hermitian structure
$\langle \cdot , \cdot \rangle$ on $\bl_\ba$,
which extends to $\bl$ canonically, if we require
orthogonality between odd and even subspaces \cite[Sec.1]{ka}.
Let $\mu_A$ be the positive Haar measure of $A$.
We assume that $G$ has a positive Haar supermeasure $\mu_G$
\cite{coulembier}\cite{all2}, i.e. $G$ has non-zero volume,
so that $\mu_X = \mu_G\, \mu_A$ is a positive measure of $X$
 for us to perform Berezin integration
\cite[\S4.6]{vsv2}.

Consider $\oplus_C W(\bl)_{\lat}$,
summed over the integral weights $\la$ in $C$ of (\ref{vgood}).
Let
\beq
 (s,t) = \int_X \langle s , t \rangle \, \mu_X \;\;,\;\;
 s,t \in \oplus_C W(\bl)_{\lat} .
 \label{conve}
 \eeq
We shall show that (\ref{conve}) is positive definite
 on the square integrable elements,
so we can define
\beq
W^2(\bl) = \mbox{completion of } \{s \in \oplus_C W(\bl)_{\lat} \;;\;
(s,s) \mbox{ converges} \} .
\label{convf}
\eeq
Then $W^2(\bl)$ is a super Hilbert space, and it
is a unitary $G \times T$-representation.
See (\ref{shil}), (\ref{unita}) and Proposition \ref{iaai}.

The next theorem
determines the occurrences of $\Theta_{\la + \rho}$ in $W^2(\bl)$.
Let $\mbox{Im}(\Phi)$
denote the image of the moment map $\Phi : X \lra i \ft^*$.
By Theorem \ref{thm1}, $\mbox{Im}(\Phi) \subset i \ft_{\rm reg}^*$,
so $\mbox{Im}(\Phi)$ is contained in an open Weyl chamber.
We choose $\om$ so that $\mbox{Im}(\Phi) \subset \widetilde{C}$
of (\ref{hccone}).

\begin{theorem}
Let $\om = i \partial \bar{\partial}F$ be a $G \times T$-invariant
pseudo-K\"ahler form on $X$, where $F$ is strictly convex.
Then we have a unitary $G \times T$-representation on the
super Hilbert space
$W^2(\bl) \cong \sum_{{\rm Im}(\Phi)} \Theta_{\la + \rho}$.
\label{thm2}
\end{theorem}

In Theorem \ref{thm2},
we identify $\fh^* \cong i\ft^*$ by Remark \ref{ide}.
Also, $\sum_{{\rm Im}(\Phi)}$ denotes the Hilbert space sum
over the integral weights $\la$ in $\mbox{Im}(\Phi)$.

Let us consider the consequences of Theorems \ref{thm1} and \ref{thm2}.
Let $C^\circ$ be the interior of $C$, given by
replacing $\geq$ with $>$ in (\ref{hccone}).
By Theorem \ref{thm1}, ${\rm Im}(\Phi) \subset C^\circ$.
So if $\la \in C \bsl C^\circ$, then $\la \not\in {\rm Im}(\Phi)$,
and by Theorem \ref{thm2}, $\Theta_{\la + \rho}$ cannot occur in $W^2(\bl)$.
To remedy this defect, we
generalize the construction of $X$.

Let $\Pi_c$ denote the compact simple roots.
Given $R \subset \Pi_c$, we define
\begin{equation}
\si = \{ \la \in C \;;\;
\la(h_\al) = 0 \mbox{ for all } \al \in R
\mbox{ and } \la(h_\al) > 0 \mbox{ for all } \al \in \Pi_c \bsl R\} .
\label{sel}
\end{equation}
We call $\si$ a cell. The subsets of $\Pi_c$ correspond to the cells,
and $C$ is a disjoint union of the cells.
Note that $R = \emptyset$ corresponds to the cell $\si = C^\circ$.

Fix a cell $\si$, and equivalently $R \subset \Pi_c$.
We define
\[ \fh_\si = \cap_R \ker(\al) \subset \fh ,\]
and by convention $\fh_\si = \fh$ if $R = \emptyset$.
Let $\ft_\si = \ft \cap \fh_\si$ and $\fa_\si = \fa \cap \fh_\si$.
We use the coadjoint action to define
$\fg^\si = \{ v \in \fg \;;\; \ad_v^* \si = 0\}$.
Its commutator subalgebra $\fg_{\rm ss}^\si = [\fg^\si, \fg^\si]$
is a semisimple Lie algebra. Let
\[ X_\si = G/G_{\rm ss}^\si \times A_\si .\]
It is a complex supermanifold (see (\ref{komp})) with $G \times H_\si$-action.
The special case $X_\si = GA$ corresponds to $\si=C^\circ$.

Let $\overline{R}$ be the roots which are nonnegative linear combinations of $R$.
Then
\[ (\fa_\si^*)_{\rm reg} =
 \{\la \in \fa_\si^* \;;\; \la(h_\al) \neq 0 \mbox{ for all } \al \in \De^+ \bsl \overline{R}\}.
 \]
This is a disjoint union of open cones in $\fa_\si^*$.
Similar to (\ref{paki}), we have $\fa_\si^* \cong i \ft_\si^*$, and we define
$(i\ft_\si^*)_{\rm reg}$ accordingly.
The following theorem generalizes Theorem \ref{thm1}.

\begin{theorem}
Every $G \times T_\si$-invariant closed (1,1)-form on $X_\si$ can be expressed as
$\om = i \pa \bp F$, where $F \in C^\infty(A_\si)$.
It is pseudo-K\"ahler if and only if $F$ is
nondegenerate and $\mbox{\rm Im}(F') \subset (\fa_\si^*)_{\rm reg}$.
The right $T_\si$-action has moment map $\Phi_\si : X_\si \lra i\ft_\si^*$ given by
$\Phi_\si(ga) = \frac{1}{2}F'(a)$.
\label{thm3}
\end{theorem}

Fix a $G \times T_\si$-invariant pseudo-K\"ahler form $\om_\si = i \pa \bp F$ on $X_\si$,
where $F$ is strictly convex.
We similarly construct a super line bundle $\bl^\si$
over $X_\si$ with Hermitian structure. Let $\cH(\bl^\si)$ denote its holomorphic sections,
and we similarly define
$\cH(\bl^\si)_\la$ for integral weight $\la \in i\ft_\si^*$.
This is a $G$-module, and we let $W(\bl^\si)_\la$ be its irreducible submodule
which contains the highest weight vector. Consider
$\oplus_\si W(\bl^\si)_\la$,
summed over the integral weights $\la$ in $\si$.
There exists a $G \times H_\si$-invariant measure $\mu_{X_\si}$ on $X_\si$.
Similar to (\ref{conve}) and (\ref{convf}),
 we integrate
 to define
\[ (s,t) = \int_{X_\si} \langle s, t \rangle \, \mu_{X_\si} \;\;,\;\;
s,t \in \oplus_\si W(\bl^\si)_\la \]
and obtain the super Hilbert space
\[ W^2 (\bl^\si) = \mbox{completion of } \{ s \in \oplus_\si W(\bl^\si)_\la \;;\;
(s,s) \mbox{ converges}\} .\]

By Theorem \ref{thm3}, ${\rm Im}(\Phi_\si) \subset (i \ft_\si^*)_{\rm reg}$,
so ${\rm Im}(\Phi_\si)$ is contained in an open Weyl chamber of $i \ft_\si^*$.
We choose $\om$ so that $\mbox{Im}(\Phi_\si) \subset \widetilde{C}$
of (\ref{hccone}).
Theorem \ref{thm2} generalizes to the following.

\begin{theorem}
Let $\om = i \partial \bar{\partial}F$ be a $G \times T_\si$-invariant
pseudo-K\"ahler form on $X_\si$, where $F$ is strictly convex.
Then we have a unitary $G \times T_\si$-representation on the
super Hilbert space
$W^2(\bl^\si) \cong \sum_{{\rm Im}(\Phi_\si)} \Theta_{\la + \rho}$.
\label{thm4}
\end{theorem}


According to Gelfand \cite{gz},
a model of compact Lie group is a unitary representation
on a Hilbert space in which every irreducible representation
occurs once.
For non-compact Lie groups with compact Cartan subgroups,
this notion is generalized to models of
holomorphic discrete series representations \cite{hc-vi},
and is constructed via geometric quantization \cite{jfa}.
We now extend this construction to the super setting,
where we obtain a model of irreducible
highest weight Harish-Chandra supermodules.

Fix a cell $\si \subset C \subset \fa^*$.
Let $\{\la_j\}$ be a basis of $\fa_\si^*$ such that
$\si = \{\sum_j c_j \la_j \;;\; c_j > 0\}$.
Define
\begin{equation}
 {\tf}_\si : \fa_\si \lra \br \;,\;
{\tf}_\si (x) = \sum_j \exp (\la_j(x)) .
\label{bebas}
\end{equation}
This is related to $F_\si \in C^\infty(A_\si)$ by Remark \ref{ide}.
We apply geometric quantization
to $(X_\si, i \partial \bar{\partial} F_\si)$
and obtain $W^2(\bl^\si)$.
Let $\sum_{\{\si\}}$ denote the sum over the collection of all the cells $\si$.

\begin{theorem}
For each cell $\si$, let $F_\si \in C^\infty(A_\si)$ be given by (\ref{bebas}).
 Then $\sum_{\{\si\}} W^2(\bl^\si)$
is a unitary $G$-representation on a super Hilbert space
in which every irreducible highest weight Harish-Chandra supermodule occurs once.
\label{thm5}
\end{theorem}

For symplectic manifolds with Lie group actions,
symplectic reduction transforms them to symplectic manifolds
of lower dimensions \cite{mw}.
We now extend this process to the super setting.

Let $\om = i \partial \bar{\partial}F$
be a $G \times T_\si$-invariant pseudo-K\"ahler form on $X_\si$,
with moment map $\Phi_\si : X_\si \lra i\ft_\si^*$ of the right $T_\si$-action.
Let $\lat \in \mbox{Im}(\Phi_\si)$, and let
\[
\iimath : \Phi_\si^{-1}(\lat) \hookrightarrow X_\si
\]
be the natural inclusion.
There is a natural right $T_\si$-action on $\Phi_\si^{-1}(\lat)$,
and it leads to the quotient map
\[
 \pi : \Phi_\si^{-1}(\lat) \lra \Phi_\si^{-1}(\lat)/T_\si.
 \]

\begin{theorem}
There exists a discrete set $\Gamma \subset A_\si$ such that
$\Phi_\si^{-1}(\lat)/T_\si = G/G^\si \times \Gamma$.
It has a unique $G$-invariant pseudo-K\"ahler form $-id \lat$ such that
$\pi^* (-i d \lat) = \iimath^* \om$.
\label{thm6}
\end{theorem}

In Theorem \ref{thm6}, $d$ is the exterior derivative
\cite[p.234]{cw}, and
we shall explain in (\ref{baz})
that $-id \lat$ can be regarded as a $G$-invariant 2-form on $G/G^\si$.
Write
\[ (X_\si)_{\lat} = \Phi_\si^{-1}(\lat)/T_\si \;\;,\;\;
\om_{\lat} = -i d\lat .\]
The process
\[ (X_\si, \om) \leadsto ((X_\si)_{\lat} , \om_{\lat}) \]
is called symplectic reduction with respect to $\lat$,
and $((X_\si)_{\lat}, \om_{\lat})$ is called the symplectic quotient.

The concept of ``quantization commutes with reduction''
is proposed by Guillemin and Sternberg \cite{gs2},
often fondly written as $[\QQ,\RR]=0$.
See \cite{sj} for its later developments.
We now show that it holds in our setting.

Let us write
$\QQ(X_\si, \om) = W^2(\bl^\si)$
to emphasize that we obtain $W^2(\bl^\si)$
from $(X_\si, \om)$ via geometric quantization.
Let $\QQ(X_\si, \om)_{\lat}$ denote the elements of $\QQ(X_\si, \om)$
that transform by $\lat$ under the right $T_\si$-action.

We similarly apply geometric quantization to $((X_\si)_{\lat}, \om_{\lat})$,
and obtain a unitary $G$-representation
$\QQ((X_\si)_{\lat}, \om_{\lat})$.
The following theorem says that applying symplectic reduction followed by
geometric quantization is equivalent to geometric quantization
followed by taking subrepresentation.

\begin{theorem}
Let $\om = i \pa \bp F$ be a $G \times T_\si$-invariant pseudo-K\"ahler form
on $X_\si$, where $F$ is strictly convex.
As unitary $G$-representations,
$\QQ ((X_\si)_{\lat}, \om_{\lat}) \cong \QQ(X_\si, \om)_{\lat}$.
\label{thm7}
\end{theorem}

We end this Introduction by stating two conventions.

\begin{remark}
\label{ide}
Conventions on $\fh^*$ and $C^\infty(A)$.

{\rm For convenience,
we define
\beq
\fh^* = \{\la : \fh \lra \bc  \;;\; \la \mbox{ is $\bc$-linear, }
\la(\ft) \subset i \br \mbox{ and } \la(\fa) \subset \br \}.
\label{keci}
\eeq
So $\fh^*$ is a real subspace of the dual space of $\fh$.
By restricting the elements of $\fh^*$ to $\ft$ and $\fa$ respectively,
we have an isomorphism
\beq
\fh^* \cong i \ft^* \cong \fa^* .\label{paki}
\eeq
We let $\la$ denote an integral weight of $\fh^*, i \ft^*, \fa^*$ in various contexts,
and they are all related by (\ref{paki}).
For example,
when Theorem \ref{thm1} writes about $\Phi(ga) = \frac{1}{2} F'(a)$,
it makes use of (\ref{paki}) because
$\Phi(ga) \in i \ft^*$ and $\frac{1}{2} F'(a) \in \fa^*$.

Let $x=(x_1,...,x_n)$ be some linear coordinates on $\fa$.
Given $H \in C^\infty(\fa)$, we can define
$dH = (\frac{\pa H}{\pa x_i}dx_i) : \fa \lra \fa^*$,
and this is independent of $x$.
The Hessian matrix $(\frac{\pa^2 H}{\pa x_i \pa x_j})$
depends on $x$, but its nondegenerate
or positive definite property is independent of $x$.

Let $F \in C^\infty(A)$. By the exponential map $\fa \cong A$,
so $F$ uniquely determines $\tf \in C^\infty(\fa)$ by $\tf(v) = F(e^v)$
for all $v \in \fa$.
We define the gradient mapping $F' : A \lra \fa^*$ by
$F'(e^v) = (dH)(v)$.
We say that $F$ is nondegenerate (resp. strictly convex) if
the Hessian matrix of $\tf$
is nondegenerate (resp. positive definite) everywhere.
}
\end{remark}

\sk

The sections in this article are arranged as follows.
In Section 2, we construct an $L^2$-structure on the holomorphic functions by Berezin integration,
and obtain a unitary representation on some highest weight supermodules.
In Section 3, we study the pseudo-K\"ahler forms and prove Theorem \ref{thm1}.
In Section 4, we perform geometric quantization and prove Theorem \ref{thm2}.
In Section 5, we generalize the above results to the setting
with cells, and prove Theorems \ref{thm3} and \ref{thm4}.
Consequently in Section 6, we construct the Gelfand model
and prove Theorem \ref{thm5}.
In Section 7, we study symplectic reduction and prove Theorems \ref{thm6} and \ref{thm7}.

\bigskip

\noindent {\bf Acknowledgements}. The authors thank Prof.
P. Grassi, J. Huerta, M. A. Lledo, V. Serganova and Dr. C. A. Cremonini,
S. D. Kwok, S. Noja for helpful discussions.
The first author is partially supported by the National Science and Technology Council
of Taiwan.
The second author is partially supported by
research grants CaLISTA CA 21109 and CaLIGOLA MSCA-2021-SE-
01-101086123.


\newpage
\section{Unitary supermodules}
\setcounter{equation}{0}

In this section, we construct an invariant
$L^2$-structure on the holomorphic functions
of the complex supermanifold $X = GA$.
As a result, we obtain unitary realizations
of highest weight Harish-Chandra supermodules of $G$.

We say that $u \in \fg$ is homogeneous if it belongs to $\fg_r$
for $r \in \{\0,\1\}$. In that case we let $|u| =r$ denote its parity.
A super Hermitian metric on a super vector space $V$ is a map
$H : V \times V \lra \bc$ which is linear (resp. anti-linear)
on the first (resp. second) entry, such that for all
non-zero homogeneous vectors $u,v$,
\beq
  \begin{array}{cl}
    \mbox{(a)}  &  H(u,v) = 0  \mbox{\ \  if  $ \, |u| \neq |v| \, $}
    \;\; \mbox{(consistent)} , \\
    \mbox{(b)}  &  H(u,v) = (-1)^{|u| \cdot |v|} \overline{H(v,u)} \;\;
\mbox{(super Hermitian symmetric)} , \\
    \mbox{(c)}  &  H(u,u) \in \br^+ \,,\, H(v,v) \in i \br^+
    \mbox{ for } u \in V_\0 \,,\, v \in V_\1 \;\; \mbox{(super positive)}.
  \end{array}
\label{shil}
\eeq
See \cite[\S4.1]{fg}.
Then $H|_{V_\0} \oplus (-iH)|_{V_\1}$ is an ordinary inner product on $V = V_\0 \oplus V_\1$.
We say that $V$ is a super Hilbert space if it is complete with respect to
this inner product.

We say that a representation of a Lie superalgebra $\fg$ on a super Hilbert space $V$
is \textit{unitary} if:
\begin{equation}
( X \cdot v,w ) + (-1)^{|X||v|}( v,X \cdot w ) =0,
\qquad X \in \fg, \quad v,w \in V
\label{unita}
\end{equation}
\cite[p.101]{vsv2}\cite[Sec.3]{cfv2}.
This is equivalent to conditions (U1), (U2) of
\cite[Sec.1]{cfv2}, as specified in Sec. 3 of the same work.
We say that a representation of the supergroup $G=(G_\zero,\fg)$ is
\textit{unitary}, if it is a unitary representation of the ordinary group
$G_\zero$ and a unitary representation for the superalgebra $\fg$, together
with compatibility conditions expressed in \cite{ccf}\cite{cfv2}.

Since we consider only unitarizable
highest weight supermodules, let us first
obtain a list of all real forms with such modules.
The list (\ref{goodlist}) indicates $\fl$ and $\gz$,
which uniquely determines the real form $\fg$ of $\fl$
\cite[Prop.5.3.2]{ka}.
Here $\fg_{2,c}$ denotes the compact real form of $G_2$.
\beq
\begin{array}{l}
\underline{A(m,n): } \;
\fs \fu(p,m-p|n) \\
\underline{B(m,n): } \;
\fs \fo(2m+1) \oplus \fs \fp(n,\br) \\
\underline{C(n): } \;
\fs \fp(n-1,\br) \oplus \fs \fo(2) \\
\underline{D(m,n): } \;
\fs \fo(2m) \oplus \fs \fp(n,\br) \;,\;
\fs \fo(2m-2,2) \oplus \fs \fp(n,\br) \;,\;
\fs \fo(2m)^* \oplus \fs \fp(n,\br) \\
\underline{F(4): } \;
\fs \fo(2,5) + \fs \fu(2) \\
\underline{G(3): } \;
\fg_{2,c} + \fs \fl(2,\br) \\
\underline{D(2,1;\al): } \;
\fs \fu(2)^2 \oplus \fs \fl(2,\br) \;,\;
 \fs \fl(2,\br)^3
\end{array}
\label{goodlist}
\eeq

\begin{proposition}
Let $\la \in C$.
The highest weight supermodule $\Theta_{\la + \rho}$ of $\fg$ is unitarizable
if and only if $\fg$ belongs to (\ref{goodlist})
(apart from exceptions in $F(4)$ and $G(3)$, see Remark \ref{apoo}).
\label{eeej}
\end{proposition}
\begin{proof}
In \cite{jakobsen}, Jakobsen provides a complete list of all the unitarizable
highest weight modules for basic classical Lie superalgebras.
We check case by case that, for all $G=(G_\0,\fg)$ with $\gz$ in (\ref{goodlist}),
our condition $\la \in C$ of (\ref{vgood})
ensures that $\Theta_{\la + \rho}$ belongs to Jakobsen's list.

We demonstrate our computation by an example,
$\fg = \mathfrak{osp}(2m+1|n)$, $m,n>0$, with $\fg_\0=\fso(2m+1)\oplus \fsp(n,\R)$.
Unlike the general positive systems in \cite{jakobsen},
we consider only the admissible positive system (\ref{howa}).

Following the notation in \cite[Ch.7]{jakobsen}, we write
for the highest weight
\[
\Lambda=(\mu_1, \dots, \mu_m,\la, \la-a_2, \dots,\la-a_n), \qquad
a_n\geq \dots\geq a_2\geq 0,\quad \mu_1 \geq \dots \geq \mu_n \geq 0 .
\]
We can see that (\ref{hccone}) is satisfied, in fact the
parameters $\mu_i$, $a_j$ are defined to make such
necessary condition true. For example
$\Lambda(\ep_1-\ep_2)=\mu_1- \mu_2\geq 0$.
We write
\[
\Lambda+\rho=(\mu_1+m-1/2, \dots, \mu_m+1/2, \lambda+n-m-1/2, \dots,
\la-a_n+1-m-1/2)
\]
and impose (\ref{vgood}). We get for one odd root:
\[
(\Lambda+\rho)(\ep_1-\de_1)=\mu_1+m-1/2+(\lambda+n-m-1/2)<0 .
\]
It gives $\lambda<1-\mu_1-n$, which ensures that $\Theta_{\la+\rho}$
is unitarizable \cite[Prop.7.4]{jakobsen}.

The remaining cases are checked in similar ways. Namely for each $\fg$ in
(\ref{goodlist}), we check that
(\ref{hccone}) and (\ref{vgood}) imply that $\Theta_{\la + \rho}$
appears in the list of unitarizable modules in \cite{jakobsen}.
The only rare exceptions occur in $F(4)$ and $G(3)$,
see remark below.
\end{proof}

\begin{remark}
Unitarizable modules for real forms of $F(4)$ and $G(3)$.

{\rm For $F(4)$ and $G(3)$, the condition $\la \in C$ of
(\ref{vgood}) is neither weaker nor stronger than
  unitarizability in \cite{jakobsen}.
  So the real forms in (\ref{goodlist}) have some supermodules which are unitarizable,
  as well as other supermodules which are not unitarizable.
  For example, we consider $G(3)$.
 We are concerned only about the case I
in \cite[p.102]{jakobsen}, since it is the only admissible positive system (\ref{howa}).
  In the notation of \cite[Prop.11.2]{jakobsen}, a short
  calculation gives (\ref{vgood})
  as $\mu<a+b-10$, but
unitarizability is under the more restrictive condition
$\mu<-3a-3b-9$; we have in both cases $a$, $b-a>0$, due to ordinary
unitarizability.
For convenience, we shall ignore these exceptions without explicitly saying so.
\label{apoo}
}\end{remark}

\sk

Proposition \ref{eeej} is consistent with \cite[Thm.6.2.1]{ns},
which provides a list of $\fg$ with no nontrivial unitary representation.

Assume for now that $G$ is an ordinary connected Lie group. We identify
$\fg$ with the left invariant vector fields of $G$, so
we similarly identify $\we^n \fg^*$ with the left invariant
$n$-forms of $G$.
Let $U \subset G$ be a closed subgroup.
We define
\begin{equation}
\we^n (\mathfrak{g}, \mathfrak{u})^*:=
\{ \om \in \we^n \mathfrak{g}^* \;;\;
\ad_v^* \om =
\imath(v)\om = 0
\mbox{ for all } v \in \mathfrak{u}\}.
\label{motiva}
\end{equation}
Here $\ad_v^* : \we^n \mathfrak{g}^* \lra \we^n \mathfrak{g}^*$
is the coadjoint map, and
$\imath(v) :  \we^n \mathfrak{g}^* \lra \we^{n-1} \mathfrak{g}^*$
is the contraction map.
The natural quotient $\pi: G \lra G/U$ intertwines with the left $G$-action, so
it has pullback
$\pi^* : \Om^n(G/U)^G \hookrightarrow \Om^n (G)^G$,
where $\Om^n(\cdot)^G$
denotes $n$-forms that are invariant under the left $G$-action.
Since $\pi^*$ is injective, we can identify $\Om^n(G/U)^G$ with
its image in $\Om^n (G)^G \cong \we^n \fg^*$,
consisting of elements that are invariant under the coadjoint map
of $\mathfrak{g}$ and annihilate tangent vectors of $\mathfrak{g}$.
This motivates (\ref{motiva}),
which gives
\[ \we^n (\mathfrak{g}, \mathfrak{u})^* \cong \Om^n(G/U)^G .\]
For $n = \dim \fg - \dim \fu$, the non-zero elements
(if they exist) in $\we^n (\mathfrak{g}, \mathfrak{u})^*$
are unique up to non-zero scalar multiple, and they
amount to $G$-invariant measures of $G/U$.
The existence of such measures is closely related to the notion
of unimodular groups.

We recall the unimodular Lie groups; see for instance \cite[VIII]{kn}.
For each $x \in G$, let
$\mbox{Ad}_x \in \mbox{Aut}(\fg)$ denote the adjoint representation,
and let $\De(x) = \det \mbox{Ad}_x$.
Then $\De : G \lra \br^+$ is a group homomorphism, called the modular function of $G$.
If $\De(x) = 1$ for all $x \in G$, we say that $G$ is unimodular.

The Haar measure $\mu_G$ is invariant
under the left action of $G$ on itself.
One can also consider a right invariant measure of $G$,
and in fact $\De \cdot \mu_G$ is right invariant \cite[Cor.8.30(c)]{kn}.
Therefore, we have two equivalent definitions for unimodular groups.

\begin{definition}
{\rm \cite[VIII-2]{kn} We say that a Lie group $G$ is unimodular if
its left invariant Haar measure is also right invariant, or equivalently
$\De(x) = 1$ for all $x \in G$.}
\label{whted}
\end{definition}

\begin{example}
{\rm \cite[Cor.8.31]{kn} An abelian Lie group is unimodular because its adjoint representation is trivial.
A compact Lie group is unimodular because the image of $\De$ is a compact subgroup of $\br^+$, and hence is
$\{1\}$. A semisimple Lie group is unimodular because its 1-dimensional representation is trivial.
A reductive Lie group is unimodular because it is generated by an abelian subgroup and a semisimple subgroup,
both of which are unimodular.
} \label{unim}
\end{example}

We now let $G$ be the real form of a contragredient Lie supergroup,
as outlined in the introduction.
Motivated by the Lie group setting,
we again adopt (\ref{motiva}) as the definition for
$\Om^n(G/U)^G \cong \we^n (\mathfrak{g}, \mathfrak{u})^*$.
It differs from the Lie algebra setting because odd vectors commute,
namely $v \we w = w \we v$ for $v,w \in \fg_\1$.
If $G$ or $\fg$ has super dimension $r|k$, we say that its dimension is $r+k$,
to ease the notation.
Let $\dim G = m$. Then $\we^m \fg^*$ is more than 1-dimensional,
for instance it contains
$f \we ... \we f$ where $f \in \fg_\1^*$.
The Haar super measure of $G$ is
given by $\tau_1 \we ... \we \tau_m \in \we^m \fg^*$,
where $\{\tau_i\}$ is a basis of $\fg^*$.
Because of $G$ invariance, this notion is equivalent to
the Haar super measures as in \cite{coulembier}\cite{all2}.

\begin{proposition}
Let $G$ be a real contragredient Lie supergroup.
Let $U$ be a unimodular subgroup of $G_\0$.
Let $n = \dim \gz - \dim \fu$ and $k = \dim \fg_\1$.
There exists a $G$-invariant measure of $G/U$
 given by
$\tau \we \nu \in \we^{n+k}(\fg,\fu)^*$,
where
$\tau \in \we^n (\fg, \fu)^*$ is extended from $\we^n(\gz, \fu)^*$
by annihilating $\fg_\1$,
and $\nu \in \we^k(\fg, \gz)^*$.
\label{ffub}
\end{proposition}
\begin{proof}
Let $\de_1 ,..., \de_m, \nu_1,..., \nu_k$ be a basis of $\fg$,
where $\de_i$ annihilates $\fg_\1$, and $\nu_i$ annihilates $\gz$.
Let $\de = \de_1 \we ... \we \de_m$ and $\nu = \nu_1 \we ... \we \nu_k$.
Then $\de \we \nu$
is a Haar super measure of $G$, so $\de$ and $\nu$ are
measures of $G_\0$ and $G_\1$ respectively.

We claim that $\nu \in \we^k(\fg, \gz)^*$.
Clearly $\imath(x) \nu = 0$ for all $x \in \gz$, so it remains to show that
\beq
\ad_x^* \nu = 0 \;,\; x \in \gz .
\label{trrr}
\eeq

Let $g \in G_\0$, and let $R_g$ denote the right action.
Then $R_g (\de \we \nu)$ is again a Haar super measure of $G$.
Such measures are unique up to scalar multiple,
so there exists $c \in C^\infty (G_\0)$ such that
\beq
 c(g) (\de \we \nu) = \mbox{Ad}_g^* (\de \we \nu)
= (\mbox{Ad}_g^* \de) \we (\mbox{Ad}_g^* \nu)
\;\;,\;\;
g \in G_\0 . \label{spba}
\eeq

Since $G$ is a real form of a contragredient Lie supergroup,
it follows that $G_\0$ is reductive \cite[\S2]{ka}. So
by Example \ref{unim}, $G_\0$ is
is unimodular, and its Haar measure satisfies $\mbox{Ad}_g^* \de = \de$.
Then (\ref{spba}) implies that $\mbox{Ad}_g^* \nu = c(g) \nu$, namely
we have a real representation
\beq
 \mbox{Ad}^* : G_\0 \lra \mbox{Aut}(\br(\nu)) .
 \label{tidc}
 \eeq
Since $G_\0$ is reductive,
it is generated by its center $Z$
and semisimple subgroup $(G_\0)_{\rm ss}$.
We have $Z \subset T$, so $Z$ is compact.
For compact groups and semisimple Lie groups,
the only real 1-dimensional representation
is the trivial representation.
Hence $\mbox{Ad}^*(Z)$ and $\mbox{Ad}^*(G_\0)_{\rm ss}$ act trivially
on $\br(\nu)$,
and so (\ref{tidc}) is a trivial representation.
This implies (\ref{trrr}). We have shown that:
\beq
\mbox{$G_\1 = G/G_\0$ has a $G$-invariant measure
given by $\nu \in \we^k(\fg, \gz)^*$.}
\label{petm}
\eeq

Let $n = \dim \gz - \dim \fu$.
Since $G_\0$ and $U$ are unimodular,
$G_\0/U$ has a $G_\0$-invariant measure \cite[Prop.8.36]{kn}.
In other words, there exists
$0 \neq \tau_\0 \in \we^n (\gz, \fu)^*$.
We extend it to $\tau \in \we^n \fg^*$ by annihilating $\fg_\1$.
We claim that
$\tau \in \we^n(\fg, \fu)^*$.
Let $v \in \fu$. It is clear that $\imath(v) \tau = 0$,
and it remains to check that $\ad_v^* \tau = 0$.
Let $x^1,...,x^n \in \fg$, and write $x^i = x_\0^i + x_\1^i$. We have
\[ \begin{array}{rll}
(\ad_v^* \tau)(x^1,...,x^n)
& = \sum_i \tau(x^1,..., [v,x^i],..., x^n) & \\
& = \sum_i \tau(x_\0^1,..., [v,x^i]_\0 ,..., x_\0^n)
& \mbox{because $\tau$ annihilates $\fg_\1$} \\
& = \sum_i \tau(x_\0^1,..., [v,x_\0^i] ,..., x_\0^n)
& \mbox{because $[v,x_j^i] \in \fg_j$} \\
& = (\ad_v^* \tau_\0)(x_\0^1,..., x_\0^n)
& \\
& = 0
& \mbox{because $\tau_\0 \in \we^n(\gz,\fu)^*$}.
\end{array}
\]
This shows that $\tau \in \we^n(\fg, \fu)^*$ as claimed.

Let $\nu \in \we^k(\fg, \gz)^*$ be given by (\ref{petm}).
Then $\tau \we \nu \in \we^{n+k}(\fg,\fu)^*$.
Since $\tau$ and $\nu$
give super measures of $G_\0/U$ and $G_\1$ respectively,
it follows that $\tau \we \nu$
gives a super measure of $G/U$.
\end{proof}



The contragredient Lie supergroup $L$ has an Iwasawa decomposition of open subset
$GAN \subset L$, where the unipotent subgroup $N$ is determined by an
admissible positive system (\ref{howa}). In this way,
we obtain the complex supermanifold
\[ X = GA \]
as an open subset of the complex homogeneous superspace $L/N$ (see \cite{cfv}).

Let $\cH(X)$ denote the space of holomorphic functions on $X$.
We first define an action on $\cH(X)$.
In the Super Harish-Chandra pairs (SHCP) formalism,
an action of $G=(G_\zero, \fg)$
is given by a pair consisting of an action of the ordinary Lie group $G_\zero$
and an action of the Lie superalgebra $\fg$, with compatibility conditions
\cite[Ch.7]{ccf}.
Since $X$ is the Cartesian product of the Lie supergroup $G$
and the ordinary Lie group $A$, it splits \cite[Ch.9]{ccf}.
So we can express a holomorphic function as
\begin{equation}
\label{split-comp}
f=\sum_I f_I\xi^I \in
\cH(G_\zero A) \otimes \wedge(\xi) \cong \cH(X) .
\end{equation}
Here we use multiple index notation, for example if $I=\{1,2\}$, then
$\xi^I = \xi^1 \xi^2$.
The action of $G_\zero$
on $f \in \cH(X)$
is given by the
action on each component $f_I$ together with
an action on $\xi^I$ preserving degree
\cite[Sec.4]{cfv}.
The action of the Lie superalgebra $\fg$ is via left invariant
vector fields.

We have $G \times H$-action (namely left $G$ and right $H$) on $X$,
because $H$ normalizes $N$.
The actions of $G$ and $T$ commute. This is true
in the ordinary setting, so the actions of $G_\zero$ and $T$ commute.
As for the super setting, in the SHCP language,
the infinitesimal action of $\fg$ is via the
left invariant vector fields,
hence it also commutes with the right action of $T$.

We now want to give a functorial point of view of these actions.
Let $\fg=\Lie(G)$, $\{\Xi^i, \, i=1, \dots, m\}$
a basis for $\fg_\one$.
We write any
$g \in G(\zR)$ as
\[
g=g_\zero(1+\xi^1\Xi^1) \dots (1+\xi^m\Xi^m), \qquad
\xi^i \in \zR_\one, g_\zero \in G_\zero(\zR) ,
\]
where $\zR$ is the
superalgebra of global sections of a supermanifold, and
$G(\zR)$ are the $\zR$ points of the Lie supergroup $G$
\cite[Props.3.2.4,3.5.2]{ga}\cite[Ch.3]{ccf}.
Hence to define an action of $G$ on $\cH(X)$,
we only need to provide a functorial action
of $g_\zero \in G_\zero(\zR)$ and $1+\xi^i\Xi^i$ on the $\zR$
points of the super vector
space $\cH(X)$, that is on:
$$
\cH(X)(\zR):= \zR_\zero \otimes \cH(X)_\zero \oplus \zR_\one \otimes \cH(X)_\one
$$
\cite[Ch.1,3,8]{ccf}.
The action of $g_\zero$ is given
by the ordinary theory, $\zR$ linearity and (\ref{split-comp}),
so we only need to specify the action of $1+\xi^i\Xi^i$:
\beq\label{act-gr}
\begin{array}{l}
(1+\xi^i\Xi^i) \cdot (rf+\rho \phi) = r f + \rho \phi +
\xi^i (r \Xi^i(f) - \rho \Xi^i(\phi) ), \\
\mbox{where }  rf \in \zR_\0 \otimes  \cH(X)_\0
\mbox{ and } \rho \phi \in \zR_\1 \otimes  \cH(X)_\1 .
\end{array}
\eeq

For an integral weight $\la \in i \ft^*$, we let $\chi : T \lra S^1$
be its character, namely $\exp(\la(v)) = \chi(\exp v)$ for all $v \in \ft$.
Then for a $G \times T$-module $V$ (such as $V = \ch(X)$), we use
$\la \in i\ft^*$ to define the $G$-subrepresentation
\beq
 V_\la = \{f \in V \;;\;
R_t f = \chi(t) f \mbox{ for all } t \in T \} ,
\label{puku}
\eeq
where $R$ is the right $T$-action.

We now consider the holomorphic functions $\cH(X)$.
By (\ref{puku}), we obtain $\cH(X)_\la \subset \cH(X)$. Let $\la \in C$,
for $C$ in (\ref{vgood}). Let
\beq
\begin{array}{rcl}
W_\la & = & \mbox{smallest SHCP representation of $(G_\0,\fg)$ in $\cH(X)_\la$} \\
& & \mbox{which contains the highest weight vector.}
\label{fupi}
\end{array}
\eeq
Then $W_\la \cong \Theta_{\la + \rho}$, where $\Theta_{\la + \rho}$
is the Harish-Chandra supermodule with highest weight $\la$ \cite[Thm.1]{cfv}. Consider
\[ \oplus_C W_\la \subset \ch(X) ,\]
summed over all the integral weights $\la \in C$.
In this way, each $W_\la$ is irreducible.
We provide an $L^2$-structure on $\oplus_C W_\la$
by Berezin integration \cite[\S4.6]{vsv2}
on the invariant form of $X$.
Let $\mu_G$, $\mu_{G_\0}$ and $\mu_A$ be the
Haar super measures of $G$, $G_\0$ and $A$ respectively,
and let $\mu_X = \mu_G \mu_A$ be their product measure on $X$.
Fix a positive real valued function $\psi \in C^\infty(A)$,
and extend it uniquely to a $G$-invariant
function on $X$, still denoted by $\psi$.
We define
\beq
(f,h) = \int_X f \bar{h} \psi \, \mu_X
\;\;,\;\; f,h \in \oplus_C W_\la . \label{hemi}
\eeq
We shall see in the next proposition that (\ref{hemi})
is positive definite on the square integrable functions,
so that we can take the completion,
\beq
 W^2 = \mbox{completion of } \{f \in \oplus_C W_\la ;\;\;
 (f,f) \mbox{ converges}\} .
 \label{hemij}
 \eeq
Clearly $W^2$ depends on $\psi$. In later section, we shall let $\psi = e^{-F}$,
where $F$ is the potential function of a pseudo-K\"ahler form.


\begin{proposition}
For all $\fg$ in (\ref{goodlist}),
$W^2$ is a unitary $G \times T$-representation with respect to
the positive definite super Hermitian metric (\ref{hemi}).
\label{veca}
\end{proposition}
\begin{proof}
Let
\[ (\oplus_C W_\la)^2 = \{f \in \oplus_C W_\la ;\;\;
 \int_X f \bar{f} \psi \, \mu_X \mbox{ converges}\} .\]
We restrict the $\fg$-action on $W_\la$ to $\ft$, and obtain
the weight space decomposition $W_\la = \oplus_a W_\la^a$.
The proof is divided into four steps:
\beq \begin{array}{cl}
\mbox{(a)} & (\oplus_C W_\la)^2 \mbox{ is a super vector space}, \\
\mbox{(b)} & \mbox{the $G \times T$-action on $\oplus_C W_\la$ preserves (\ref{hemi})},\\
\mbox{(c)} & \mbox{for any $G \times T$-invariant super Hermitian form $H$ on $W_\la \oplus W_\nu$}, \\
& \mbox{we have $H(W_\la,W_\nu) = H(W_\la^a, W_\la^b) = 0$ for $\la \neq \nu$ and $a \neq b$,} \\
\mbox{(d)} & \mbox{the $G$-invariant Hermitian forms on each $W_\la$ are unique
up to scalar.}
\end{array}
\label{pdu}
\eeq

We apply Proposition \ref{ffub} with $U=\{e\}$ and
express the Haar super measure of $G$ as
$\mu_G = \mu_{G_\0} \, \mu_{G_\1}$, where $\mu_{G_\0}$
is the Haar measure of $G_\0$, and $\mu_{G_\1}$
is the $G$-invariant measure of $G_\1 = G/G_\0$.
We take its product with the Haar measure $\mu_A$ of $A$,
and obtain a
$G \times A$-invariant measure of $X$ by
$\mu_X = \mu_G \, \mu_A = \mu_{G_\0 A} \, \mu_{G_\1}$,
where $\mu_{G_\0 A} = \mu_{G_\0} \, \mu_A$.
Let $P$ be the power set of $\{1,..., \dim \fg_\1\}$.
Since $\mu_X$ is a real measure,
there exists $S \subset P$ and non-zero constants $\{a_K\}_{K \in S}$ such that
\beq \mu_X = \mu_{G_\0 A} \, \mu_{G_\1} = \mu_{G_\0 A} (\sum_{K \in S} a_K \xi^K \bar{\xi}^K) .
\label{muzz}
\eeq

We first prove (\ref{pdu})(a).
For $I \in P$, we let $I^c$ be its complement, for example
$\{1,2\}^c = \{3,..., \dim \fg_\1\}$.
For $S$ given by (\ref{muzz}), let
$S^c = \{I^c \;;\; I \in S\}$.
By (\ref{split-comp}), we write
$f = \sum_{I \in P} f_I \xi^I \in \oplus_C W_\la$.
By (\ref{hemi}) and (\ref{muzz}),
\beq
 (f,f) = \int_{G_\1} \int_{G_\0 A} f \bar{f} \psi \, \mu_{G_\0 A} \mu_{G_\1}
= \sum_{K \in S} \sum_{I \in P} \int_{G_\1}
\int_{G_\0 A} f_I \bar{f}_I \xi^I \bar{\xi}^I \psi \, \mu_{G_\0 A} \,
 a_K \xi^K \bar{\xi}^K .
\label{sxff}
\eeq
The Berezin integration annihilates all the monomials $\xi^I \bar{\xi}^J$
other than $\xi_1 ... \xi_m \bar{\xi}_1 ... \bar{\xi}_m$
\cite[\S4.6]{vsv2}. So for each $K \in S$, the only $I$ in (\ref{sxff})
with non-trivial Berezin integration is $I = K^c \in S^c$,
and in that case we write
$c_I = \int_{G_\1} a_K \xi^I \bar{\xi}^I \xi^K \bar{\xi}^K \neq 0$.
Then (\ref{sxff}) becomes
\beq
 (f,f)
 = \sum_{I \in S^c} c_I \int_{G_\0 A} f_I \bar{f}_I \psi \,
 \mu_{G_\0 A} \;\;,\;\; c_I \neq 0.
 \label{leng}
 \eeq
 Let
$\cH^2(G_\0 A , \psi) =\{h \in \ch(G_\0 A) \;;\; \int_{G_\0 A}
 h \bar{h} \psi \, \mu_{G_\0 A} < \infty\}$.
 By (\ref{leng}),
\beq
f \in (\oplus_C W_\la)^2 \; \Longleftrightarrow \; f_I \in \cH^2(G_\0 A, \psi)
\mbox{ for all } I \in S^c .
\label{adam}
\eeq
By (\ref{adam}), since $\cH^2(G_\0 A, \psi)$ is a vector space,
so is $(\oplus_C W_\la)^2$.
This proves (\ref{pdu})(a).

Next we prove (\ref{pdu})(b).
Let $f,h \in \oplus_C W_\la$ and $X \in \fg$.
By the property of super Haar measure \cite[Lem.2.2]{coulembier},
\beq
\begin{array}{rl}
( Xf,h ) + (-1)^{|X||f|} ( f,Xh ) &=
\int_{GA} \left[ (X(f)\overline{h}+(-1)^{|X||f|}f \overline{X(h)} \right]
\psi \, \mu_X  \\
&=  \int_{GA} X(f\overline{h}) \psi \, \mu_X =0 .
\end{array}
\label{agek}
\eeq
The representation of $\fg$ satisfies the conditions (U1) and (U2)
of \cite[Thm.2]{cfv2}, so the action of $G$ on $W$ preserves (\ref{hemi}).
Since $G_\0$ is unimodular, its Haar measure
is invariant under the right $T$-action, so
the right $T$-action preserves (\ref{hemi}).
This proves (\ref{pdu})(b).

Next we prove (\ref{pdu})(c).
Let $H$ be a $G \times T$-invariant super Hermitian form on $W_\la \oplus W_\nu$.
We first show that $H(W_\la,W_\nu)=0$.
Let $\chi_\la, \chi_\nu : T \lra S^1$ be the characters of $\la,\nu$.
Since $\la, \nu$ are distinct, there exists $t \in T$ such that
$\chi_\la(t) \overline{\chi_\nu(t)} \neq 1$.
Let $R$ denote the right $T$-action.
By (\ref{fupi}), the elements of $W_\la$ and $W_\nu$
transform by $\chi_\la$ and $\chi_\nu$ respectively under $R$.
So for $f \in W_\la$ and $h \in W_\nu$,
\beq
 H(f,h) = H(R_t f, R_t h) = H(\chi_\la(t) f, \chi_\nu(t) h)
= \chi_\la(t) \overline{\chi_\nu(t)} H(f,h) .
\label{ikut}
\eeq
Since $\chi_\la(t) \overline{\chi_\nu(t)} \neq 1$, this implies that
$H(f,h) =0$. Hence $H(W_\la, W_\nu)=0$.

Next we show that $H(W_\la^a,W_\la^b) = 0$.
We restrict the $\fg$-action on $W_\la$ to $\ft$,
and let $\pi$ be the corresponding $T$-action.
We repeat the method of (\ref{ikut}),
namely if $f \in W_\la^a$ and $h \in W_\la^b$, then
$H(f,h) = H(\pi_t f, \pi_t h) = \overline{\chi_a(t)} \chi_b(t) H(f,h)$ vanishes.
This proves (\ref{pdu})(c).

Finally we prove (\ref{pdu})(d).
The highest weight space of $W_\la = \oplus_a W_\la^a$ is $W_\la^\la$.
It is even and satisfies
$\dim W_\la^\la = 1$.
Let $H_1,H_2$ be $G$-invariant super Hermitian forms on $W_\la$.
We multiply $H_1$ by a scalar so that $H_1 = H_2$ on $W_\la^\la$.
Let $0 \neq v \in W_\la^\la$, and consider all vectors of the form
$X_1 ... X_n v$ where $X_j \in \fg$. For $j=1,2$, we have
\beq
H_j(X_1 ... X_n v, Y_1 ... Y_m v) = (-1)^n H_j(v, X_n ... X_1 Y_1 ... Y_m v) .
\label{nkes}
\eeq
Let $\rho : W_\la \lra W_\la^\la$ be the projection which annihilates all the weight spaces
other than $W_\la^\la$.
By (\ref{pdu})(c), different weight spaces are orthogonal,
so the last expression of (\ref{nkes}) is
$(-1)^n H_j(v, \rho (X_n ... X_1 Y_1 ... Y_m v))$.
It has the same value for $j=1,2$, because $H_1 = H_2$ on $W_\la^\la$.
So (\ref{nkes}) has the same value for $j=1,2$.
Since $W_\la$ consists of all linear combinations of
$\{X_1 ... X_n  v \;;\; X_j \in \fg\}$, it follows that $H_1 = H_2$ on $W_\la$.
This proves (\ref{pdu})(d).

Let $H$ be the Hermitian form (\ref{hemi}).
By (\ref{pdu})(b), the $G$-action preserves $H$,
so each irreducible subrepresentation $W_\nu$ satisfies $W_\nu \subset (\oplus_C W_\la)^2$
or $W_\nu \cap (\oplus_C W_\la)^2 = 0$. So $(\oplus_C W_\la)^2 = \oplus_D W_\la$
for some $D \subset C$.

Let $\la \in D$.
Since $\fg$ belongs to the list (\ref{goodlist}),
by Proposition \ref{eeej}, $W_\la$ is a unitarizable $\fg$-module.
So together with (\ref{pdu})(d), there exists $c \in \bc$ such that $cH$
is positive definite on $W_\la$. But by (\ref{leng}), we have $H(f,f) > 0$
for all ordinary non-zero functions $f$, so $c$ is a positive real number,
namely $H$ is already positive definite on $W_\la$.
By (\ref{pdu})(c), distinct summands of $\oplus_D W_\la$
are orthogonal in $H$,
so $H$ is positive definite on $\oplus_D W_\la$.
We take its completion with respect to $H$,
and obtain the unitary representation $W^2$.
\end{proof}


\newpage
\section{Pseudo-K\"ahler structures}
\setcounter{equation}{0}

In this section, we prove Theorem \ref{thm1}.
Let $\fag$ be a complex contragredient Lie superalgebra.
Let $\fh$ be a Cartan subalgebra of $\fag$,
and let $\fag = \fh + \sum_\De \fag_\al$ be the
root space decomposition.
We have fixed a nondegenerate bilinear form in (\ref{coroot}).
It identifies $\fag$ with $\fag^*$,
so it allows us to write $\fh^* \subset \fag^*$ and so on.

Similar to (\ref{motiva}) and thereafter,
the left invariant $n$-forms on $L$ are
$\Omega^n(L)^L = \we^n \fl^*$.
We consider the elements which are also invariant under the right $H$-action.
Since the simultaneous left and right actions amount to the adjoint action,
they are invariant under the adjoint action of $H$. So
the $L \times H$-invariant forms on $L$ are
\begin{equation}
 \Om^n(L)^{L \times H} = \we_\fh^n \fag^* =
\{ \om \in \we^n \fag^* \;;\; \ad_v^* \om = 0 \mbox{ for all }
v \in \fh\} .
\label{mox}
\end{equation}
Here $\ad^* : \fl \lra \mbox{End}(\we^n \fl^*)$
is the coadjoint action.
We have a chain complex under the exterior derivative
$d : \we_\fh^n \fag^* \lra \we_\fh^{n+1} \fag^*$ \cite[p.234]{cw},
and we let $H^n(\we_\fh^\bullet \fag^*)$ denote the
cohomology at degree $n$.

Let $\fg$ be a real form of $\fl$, and suppose that $\fg$
has a compact Cartan subalgebra $\ft = \fg \cap \fh$.
We similarly define the chain complex $\we_\ft^\bullet \fg^*$.

\begin{proposition}
$H^2(\we_\fh^\bullet \fag^*) = H^2(\we_\ft^\bullet \fg^*) =0$.
\label{zet}
\end{proposition}
\begin{proof}
Let $\om \in \we_\fh^2 \fag^*$.
We first claim that for all $\al + \be \neq 0$,
\beq
\om(\fh, \fag_\al) = \om(\fag_\al, \fag_\be) =0 .
\label{bee}
\eeq
Let $x_\al \in \fag_\al$ and $x_\be \in \fag_\be$.
Pick $z \in \fh$ such that $(\al + \be)(z) \neq 0$. Then
\[ \al(z) \om(x_\al, x_\be) = \om([z,x_\al], x_\be)
= - \om(x_\al, [z,x_\be]) = - \be(z) \om(x_\al, x_\be) .\]
This shows that $\om(\fag_\al, \fag_\be) =0$.
By similar arguments (regard $\fh$ as $\fg_0$),
we also have $\om(\fh, \fag_\al) = 0$.
This proves (\ref{bee}) as claimed.

The coroot (\ref{coroot}) can be written as
$h_\al = [x_\al, x_{-\al}]$, where $x_{\pm \al} \in \fl_{\pm \al}$.
Let $\Pi$ be a simple system of $\fag$.
Since $\{h_\al \;;\; \al \in \Pi\}$
is a basis of $\fh$, we can define $\phi \in \fh^*$ by
\beq
\phi(h_\al) =
 \phi([x_\al, x_{-\al}]) = - \om (x_\al, x_{-\al})
\mbox{  for all  } \al \in \Pi .
\label{bas}
\eeq

To simplify notations, we let $\om_1 = \om$ and $\om_2 = d \phi$.
By (\ref{bee}) and (\ref{bas}),
\beq
 (\om_1)|_{\sum_{\pm \Pi} \fag_\al}
= (\om_2)|_{\sum_{\pm \Pi} \fag_\al} .
\label{taba}
\eeq

Clearly $\om_2$ is closed.
Suppose that $\om_1$ is also closed.
We now prove that $\om_1 = \om_2$.
By (\ref{bee}), we need to show that
for all positive roots $\al$,
\beq
\mbox{(a)} \, (\om_1)|_\fh = (\om_2)|_\fh \;,\;
\mbox{(b)} \, (\om_1)|_{\fag_\al + \fag_{-\al}}
= (\om_2)|_{\fag_\al + \fag_{-\al}} .
\label{stn}
\eeq

Since $d \om_j = 0$, there exist $a, b = \pm 1$
such that for all $\al,\be \in \Pi$,
\beq
\begin{array}{rl}
\om_j (h_\al,h_\be) & = \om_j ([x_\al, x_{-\al}], [x_\be, x_{-\be}]) \\
& = a \om_j ([x_\al, [x_\be, x_{-\be}]], x_{-\al})
+ b \om_j ([x_{-\al}, [x_\be, x_{-\be}]], x_\al) \\
& = (-1)^{|x_\al|} (a+b)
\al([x_\be, x_{-\be}]) \om_j(x_{-\al}, x_\al) .
\end{array}
\label{gua}
\eeq
The values of $a,b$ depend only on the parities of
$\al$ and $\be$, so they are the same for $j=1,2$.
By (\ref{taba}), the last expression of (\ref{gua}) is the same
for $j=1,2$. This proves (\ref{stn})(a).

It remains to show (\ref{stn})(b).
The height of $\al \in \De^+$ is the positive integer $h$
such that $\al = \al_1 + ... + \al_h$ and $\al_i \in \Pi$.
Let $\Pi = S_1 \subset ... \subset S_n = \De^+$, where $S_i$
consists of positive roots of heights $\leq i$.
We already know that (\ref{stn})(b) holds for all $\al \in \Pi = S_1$.
To complete the induction, suppose that it holds for all $\al \in S_i$.
An element of $S_{i+1}$ can be expressed as $\al + \be$, where
$\al, \be \in S_i$. Let $[x_\al, x_\be] \in \fag_{\al+\be}$
and $[x_{-\al}, x_{-\be}] \in \fag_{-\al-\be}$.
Since $d \om_j = 0$,
\beq
 \om_j([x_\al,x_\be],[x_{-\al},x_{-\be}])
= a \om_j ([x_\al, [x_{-\al},x_{-\be}]], x_\be)
+ b \om_j ([x_\be, [x_{-\al},x_{-\be}]], x_\al) ,
\label{pon}
\eeq
where $a,b$ are the same for $j=1,2$.
By the induction assumption, the right hand side of (\ref{pon})
are the same for $j=1,2$.
so $\om_1$ and $\om_2$ agree on $\fag_{\al + \be} + \fag_{-\al-\be}$.
This completes the induction, which proves (\ref{stn})(b).
Together with (\ref{bee}), they imply that $d \phi = \om$. We have proved that
$H^2(\we_\fh^\bullet \fag^*) =0$.

Since $\we_\fh^\bullet \fag^*$ is the complexification of
$\we_\ft^\bullet \fg^*$, by the universal coefficient theorem
on chain complexes, it implies that
$H^2(\we_\ft^\bullet \fg^*) =0$.
The proposition follows.
\end{proof}

We are interested in $X = GA$.
By the Iwasawa decomposition, it is a complex supermanifold
via the embedding $GA \hookrightarrow L/N$.
The $G \times T$-invariant $(1,0)$ (resp. $(0,1)$) forms
are the complex 1-forms that are eigenvectors
of the complex structure on $GA$
with eigenvalue $i$ (resp. $-i$), and their exterior products
lead to the $G \times T$-invariant $(p,q)$-forms
$\Om^{p,q}(GA,\bc)^{G \times T}$.
The Dolbeault operators $\pa, \bp$
acts on them, where $\pa$
raises the $p$-degree and $\bp$ raises the $q$-degree.

\begin{proposition}
Every closed element of $\Om^{1,1}(GA)^{G \times T}$
can be written as $\om = i \pa \bp F$ for some $F \in C^\infty(A)$.
\label{kam}
\end{proposition}
\begin{proof}
The $G \times T$-invariant forms on $GA$ are the
product of chain complexes,
\beq \Om^\bullet (GA)^{G \times T}
= \Om^\bullet(A) \otimes \Om^\bullet(G)^{G \times T}
= \Om^\bullet(A) \otimes \we_\ft^\bullet \fg^*  , \label{niti}
\eeq
where $\we_\ft^\bullet \fg^*$ is defined similarly as (\ref{mox}).
Since $\Om^\bullet(A)$ has trivial cohomology,
and since $H^2(\we_\ft^\bullet \fg^*) =0$ by Proposition \ref{zet},
the chain complex (\ref{niti}) at degree 2 gives
\beq
H^2(\Om^\bullet(GA)^{G \times T}) =0 .
\label{far}
\eeq

Given $f \in \fg^*$, it satisfies $\ad_v^* f =0$ for all $v \in \ft$
if and only if $f \in \ft^*$.
Hence $\we_\ft^1 \fg^* = \ft^*$.
The chain complex (\ref{niti}) at degree 1 gives
\beq \begin{array}{rl}
\Om^1 (GA)^{G \times T}
& = \Om^1(A) + C^\infty(A) \otimes \we_\ft^1 \fg^* \\
& = C^\infty(A) \otimes \fa^* + C^\infty(A) \otimes \ft^* \\
& = C^\infty(A) \otimes \fh^* .
\end{array}
\label{car}
\eeq

Let $\om$ be a
closed element of $\Om^{1,1} (GA)^{G \times T}$.
By (\ref{far}) and (\ref{car}), we can write
\beq
\om = d \be \;\;,\;\; \be \in C^\infty(A) \otimes \fh^* .
\label{kanx}
\eeq
Write $\be = \psi + \bar{\psi}$,
where $\psi \in C^\infty(A, \bc) \otimes \we^{0,1}\fh^*$.
Since $\om$ is a $(1,1)$-form, it follows that
$\bp \psi = 0$.
Let $z = x + iy$ be the holomorphic coordinates
of $H$, where $x$ and $y$ are the coordinates on $A$ and $T$
respectively. Write
$\psi = \sum_k f_k \, d \bar{z}_k$, where $f_k \in C^\infty(A,\bc)$.
Then
\beq
0 =  \bp \psi = \frac{1}{2} \sum_{j,k} \frac{\pa f_k}{\pa x_j} d \bar{z}_j
\we d \bar{z}_k .
\label{bagu}
\eeq
By (\ref{bagu}),
$\frac{\pa f_k}{\pa x_j} = \frac{\pa f_j}{\pa x_k}$ for all $j,k$.
Hence there exists $h \in C^\infty(A,\bc)$ such that $\frac{\pa h}{\pa x_k} = f_k$.
Let $F = -2i(h - \bar{h}) \in C^\infty(A)$.
Then
\[ i \pa \bp F = 2 \pa \bp(h - \bar{h}) = \pa \psi + \bp \bar{\psi}
= d (\psi + \bar{\psi}) = d \be = \om .\]
This proves the proposition.
\end{proof}

We recall the symplectic action and moment map for ordinary
symplectic manifold $(X_\zero,\om_\zero)$ \cite[\S26]{gs}.
Suppose that a Lie group $T$ acts on $X_\0$.
Originally, its moment map is written as $X_\zero \lra \ft^*$.
But it is more convenient to work with $i \ft^*$,
so we add $i$ and write
\begin{equation}
\Phi_\zero : X_\zero \lra i\ft^* .
\label{padi}
\end{equation}
We say that $\Phi_\zero$ is
$T$-equivariant if it intertwines the $T$-action on $X_\zero$
with the coadjoint action on $i\ft^*$.

We recall the construction of $\Phi_\0$.
Let $\xi \in \ft$.
Let $\xi^\sh$ denote the infinitesimal vector field on $X_\zero$,
given by $(\xi^\sh f)(x) = \frac{d}{dt}|_{t=0} f(e^{t \xi}x)$ for all
functions $f$ and $x \in X_\zero$.
Let $\imath(\xi^\sh)\om_\zero$ be the 1-form
given by $(\imath(\xi^\sh)\om_\zero)(v) =
\om_\zero(\xi^\sh,v)$
for all vector fields $v$.
We say that the $T$-action on $X_\zero$ has moment map
(\ref{padi}) if $\Phi_\0$ is $T$-equivariant and
\beq d(\Phi_\zero, \xi) = i \, \imath(\xi^\sh) \om_\zero .
\label{pqma}
\eeq
Here $(\Phi_\zero,\xi) \in C^\infty(X_\zero)$ denotes the function
$x \mapsto (\Phi_\zero(x))(\xi)$, so that $d(\Phi_\zero,\xi)$
is a 1-form on $X_\zero$ with imaginary values
(due to the factor $i$ in (\ref{padi})).

An action on a symplectic manifold may not have moment map.
An obstruction to (\ref{pqma})
is the cohomology $H^1(X_\zero)$, so for example the
ordinary 2-torus acting on itself
preserving the invariant volume form has no moment map.

A sufficient condition for (\ref{pqma})
is that $\om_\zero = d \be_\zero$,
and $\be_\zero$ is also $T$-invariant. In that case the moment map is given by
\[
(\Phi_\zero(x))(\xi) = - i\be_\zero(\xi^\sh)_x \;\;;\;\;
x \in X_\zero \;,\; \xi \in \ft .
\]

We now turn to the supermanifold $X$.
A symplectic form on $X$ is a closed nondegenerate 2-form,
and a pseudo-K\"ahler form is a symplectic $(1,1)$-form.
We want to define the
moment map of the $T$-action on the symplectic supermanifold $(X,\om)$,
\begin{equation}
\Phi: X \lra i\ft^* .
\label{padi-super}
\end{equation}
Denote, as usual
$\Phi^*:\cO(i\ft^*) \lra \cO(X)$ the map on the superalgebras
of global sections, induced by $\Phi$
(see \cite[Ch.4]{ccf}  \cite[Ch.6]{kessler}).
Let $(\Phi,\xi):=\Phi^*(\xi) \in \cO(X)$, $\xi \in \ft$, viewed as
an element of $\cO(i\ft^*)$. We say that $\Phi$ as in (\ref{padi-super})
is the \textit{moment map} if it is $T$-equivariant and
\beq\label{sup-mom}
d(\Phi, \xi) =
i \imath(\xi^\sh) \om .
\eeq

\begin{proposition}
\label{iste-supe}
Let $T$ be a Lie group acting on a supermanifold $X$.
Let  $\om = d \be $ be a symplectic form on $X$,
where $\be$ is $T$-invariant. Then
$(\Phi,\xi) = - i\be(\xi^\sh)$ defines a moment map.
\end{proposition}
\begin{proof}
We first observe that, since $i\ft^*$ is an ordinary manifold,
by the Chart theorem (see \cite[Thms.4.1.11,4.2.5]{ccf}),
$\Phi$ is determined by the choice of rk($\ft$) {\sl even} functions
in $\cO(X)$ the global sections on $X$. Since $\xi \in  \cO(i\ft^*)$
are generators, $(\Phi,\xi) = - i\be(\xi^\sh)$ defines a map $\Phi: X\lra i\ft$.
We now prove (\ref{sup-mom}). It
is the same as for the ordinary setting; we
apply the Lie derivative $L_{\xi^\sh} =
\imath(\xi^\sh)d + d \imath(\xi^\sh)$
to $\be$, obtaining $d (\be(\xi^\sh)) = \imath(\xi^\sh) \om$
(see for example
\cite{kessler}
for the Lie derivative in super setting).
\end{proof}

\begin{remark}\label{fopts}
Functor of points notation.

 {\rm
In what follows, to ease the notation, we shall employ the
functor of points notation for the moment map, writing
$\Phi(ga)$ for $g \in G(R)$ and $a \in A(R)$ for an arbitrary
supermanifold $R$. For more details see \cite[\S3.2]{ccf}.
Also, since $G$ acts on $GA$, we can define $g \cdot \Phi$ as
in \cite[p.35]{cfv}, via SHCP terminology.
} \end{remark}

\sk

\noindent {\it Proof of Theorem \ref{thm1}:}

The first statement of Theorem \ref{thm1} follows from Proposition \ref{kam}.
We now study the conditions for $\om$ to be nondegenerate.
Since $\om$ is $G$-invariant,
it suffices to consider $\om_a$ for $a \in A$.
By (\ref{kanx}), we have $\om = d \be$, where
$\be \in C^\infty(A) \otimes \fh^*$.
Since $d \fa^* = 0$, we may assume that
\beq
\be = \sum_k f_k \xi_k \in C^\infty(A) \otimes \ft^* \;\;,\;\;
 \om = d \be = \sum_k df_k \we \xi_k + \sum_k f_k d \xi_k ,
 \label{tiu}
 \eeq
 where $f_k \in C^\infty(A)$ and $\xi_k \in \ft^*$.

Since $\ft$ is a compact Cartan subalgebra of $\fg$, we have the
root space decomposition
\beq
 \fg = \ft + \sum_{\De^+} \fg_\al \;\;,\;\;
\fg_\al = \fg \cap (\fl_\al + \fl_{-\al}) . \label{haba}
\eeq
Write $\fg + \fa = \fh + V$, where $V = \sum_{\De^+} \fg_\al$.
In (\ref{tiu}), $\sum_k df_k \we \xi_k \in \Om^1(A) \otimes \ft^*$
annihilates $V$. Also, $\sum_k f_k d \xi_k$ annihilates $\fh$.
Hence for each $a \in A$, $\om_a$ is nondegenerate if and only if
the following restricted 2-forms are both nondegenerate,
\beq
\mbox{(a)} \; (\sum_k (df_k)_a \we \xi_k)|_\fh \;\;,\;\;
\mbox{(b)} \; (\sum_k f_k(a) d \xi_k)|_V .
\label{bntt}
\eeq

We identify $A$ with $\fa$, and $F$ with $\tf$, see Remark \ref{ide}.
Write $z= x+ iy$, where
$x$ and $y$ are linear coordinates on $\fa$ and $\ft$ respectively.
By Proposition \ref{kam} and (\ref{tiu}),
\begin{equation}
\be = \sum_k f_k \xi_k = \frac{1}{2} \sum_k \frac{\pa \tf}{\pa x_k} dy_k .
\label{polii}
\end{equation}
Thus (\ref{bntt})(a) is nondegenerate
for all $a \in A$
if and only if $\{f_k(a) \xi_k\}$ is a basis of $\ft^*$,
or equivalently $\sum_k f_k \xi_k : A \lra \ft^*$ is a local diffeomorphism,
namely the Hessian matrix of $\tf$
is nondegenerate everywhere.

Next we consider (\ref{bntt})(b).
Given a basis $v_\al , w_\al$ of $\fg_\al$, we have
$[v_\al,w_\al] = c t_\al \in \ft$ for some $c \in \br^\times$. Then
\beq
 (\sum_k f_k(a) d\xi_k)(v_\al, w_\al)
= (\sum_k f_k(a) \xi_k, [v_\al,w_\al])
= \be_a(c t_\al) .
\label{kka}
\eeq
So (\ref{kka})
is nonzero
for all $\al \in \De^+$
if and only if $\be_a \in \ft_{\rm reg}^*$,
or equivalently $\mbox{Im}(F') \in \fa_{\rm reg}^*$
due to (\ref{polii}).

We compute the moment map
$\Phi : X \lra i\ft^*$ of the right $T$-action.
Since $\om$ is $G$-invariant, we have
$\Phi(gx) = \Phi(x)$ (see Remark \ref{fopts}).
By Proposition \ref{iste-supe} and (\ref{polii}),
\beq
 (\Phi(x))(\xi) = - i(\be, \xi^\sh)(x) = \frac{i}{2} \sum_k \frac{\pa \tf}{\pa x_k} dy_k(\xi) =
\frac{1}{2} (F'(x))(\xi) .
\label{sima}
\eeq
Here we make use of $i \ft^* \cong \fa^*$ (see (\ref{paki})), as
 the images of $\Phi$ and $F'$ lie in $i \ft^*$ and $\fa^*$ respectively.
This proves Theorem \ref{thm1}.$\hfill$$\Box$


\newpage
\section{Geometric quantization}\label{geoqu-sec}
\setcounter{equation}{0}

In this section, we prove Theorem \ref{thm2}.
Let $\om = i \pa \bp F$ be a pseudo-K\"ahler form on $X = GA$
as given by Theorem \ref{thm1}, and suppose that $F$ is strictly convex.
Let $\omz$ be restriction of $\om$ to the ordinary manifold $\Gz A$.
We first recall geometric quantization on $\Gz A$ \cite{ko}.
Let $\blz$ be the complex line bundle on $\Gz A$ whose Chern class
is the cohomology class $[\omz]$.
Since $\om$ is exact, so is $\omz$,
hence $\blz$ is topologically a trivial line bundle.
It has a connection $\nabla$ whose curvature is $\omz$,
along with an invariant Hermitian structure $\langle \cdot, \cdot \rangle_\zero$.
Let $\cH(\blz)$ denote the holomorphic sections on $\blz$.
The $G_\zero \times T$-action on $G_\zero A$ lifts to a
$G_\zero \times T$-representation on $\cH(\blz)$.
There exists a $G_\zero \times T$-invariant section $u_\zero \in \cH(\blz)$
such that (see \cite[Prop.4.2]{jfa})
\begin{equation}
 \langle u_\zero , u_\zero \rangle_\zero = e^{-F} .
 \label{kec}
 \end{equation}

Since $X = GA$ is a Lie supergroup, it is globally split, so
$\cH(X)\cong \cH(X_\zero) \otimes \wedge(\xi)$. If
we fix such an isomorphism,
the line bundle $\blz$ on $\Gz A$ extends to the $1|0$ vector bundle
$\bl$ on $X$, (see \cite[Ch.4]{ccf}, \cite[Ch.6]{kessler})
and we can view
$\cH(\blz)$ as a subalgebra of $\cH(\bl)$,
the holomorphic sections on $\bl$.
The above section $u_\zero$ extends uniquely to
$u \in \cH(\bl)$.
Since $u_\zero$ is nowhere vanishing, so is $u$.
Therefore, it gives an identification
between $\cH(\bl)$ and the free ${1|0}$ module $\cH(X)$
of holomorphic functions on $X$, i.e. the
global sections of the structural sheaf on the complex supermanifold $X$:
\beq\label{identvb}
\cH(X)  \lra  \cH(\bl) \;,\;
f  \mapsto  f u  .
\eeq
By (\ref{kec}) and (\ref{identvb}), we extend the Hermitian structure of
$\cH(\blz)$ to a super Hermitian structure on $\cH(\bl)$ by
\beq
\langle f u,  g u \rangle := f \bar{g} \, e^{-F} ,
 \label{speci}
\eeq
where $f,g  \in \cH(X)$ have the same parity.

 We shall always let $u \in \cH(\bl)$ denote the above section.
The $G_\zero \times T$-action on $\cH(\blz)$ keeps $u_\zero$ invariant.
We extend this to a $G \times T$-action on $\cH(\bl)$ such that
$u$ is invariant.
In the same fashion as (\ref{act-gr}), we define an action of $G$ on $\cH(\bl)$
via functor of points; this amounts to an action of $G(\zR)$ on
\[
\begin{array}{l}
\cH(\bl)(\zR):= (\zR_\zero \otimes \cH(\bl)_\zero)
\oplus (\zR_\one \otimes \cH(\bl)_\one) , \\
g \cdot (r f u +\rho \phi u) :=  (g \cdot rf + g \cdot \rho \phi)u
\mbox{ where } rf \in \zR_\zero \otimes \cH(\bl)_\zero \;,\;
\rho \phi \in \zR_\one \otimes \cH(\bl)_\one .
\end{array}
\]

Let $\la \in i \ft^*$ be an integral weight, and define
$\cH(\bl)_\la$ with respect to the right $T$-action; see (\ref{puku}).
Let $W(\bl)_\la \subset \cH(\bl)_\la$ be smallest SHCP representation
of $G = (G_\0,\fg)$ in $\cH(\bl)_\la$ which contains the highest weight vector.
Consider
\[ \oplus_C W(\bl)_\la ,\]
summed over the integral weights $\la \in C$ of (\ref{vgood}).
Let $\mu_X = \mu_G \, \mu_A$ be the product of Haar super measures.
We define
\beq\label{hatt}
(s,t) =
\int_X \langle s,t \rangle \, \mu_X \;\;,\;\; s,t \in \oplus_C W(\bl)_\la .
\eeq
Similar to (\ref{hemij}), we can take the following completion because
(\ref{hatt}) is positive definite on the square integrable sections,
\[ W^2(\bl) = \mbox{completion of }
\{s \in \oplus_C W(\bl)_\la \;;\; (s,s) \mbox{ converges}\} .\]

In (\ref{hemij}), $W^2$ depends on a function $\psi$
of (\ref{hemi}).
Let us rewrite it as $W^2(e^{-F})$ to emphasize that $\psi = e^{-F}$.

\begin{proposition}
$W^2(\bl)$ is a unitary $G \times T$-representation, and $W^2(e^{-F}) \cong W^2(\bl)$
as $G \times T$-modules. \label{iaai}
\end{proposition}
\begin{proof}
By (\ref{hemi}), (\ref{speci}) and (\ref{hatt}), for all $f \in W^2(e^{-F})$,
\[ (f,f) = \int_X f \bar{f} e^{-F} \, \mu_X
= \int_X \langle f u , f u \rangle \, \mu_X
= (fu,fu) . \]
The nowhere vanishing section $u$ is $G \times T$-invariant,
so the trivialization $f \mapsto fu$ is a $G \times T$-equivariant
isometry $W^2(e^{-F})\cong W^2(\bl)$.
By Proposition \ref{veca}, $W^2(e^{-F})$ is a unitary $G \times T$-representation,
and so is $W^2(\bl)$.
\end{proof}

\sk

\noindent {\it Proof of Theorem \ref{thm2}:}

In (\ref{muzz}), $\mu_X$ determines
a subset of the power set of $\{1,..., \dim \fg_\1\}$, denoted by $S$.
By Proposition \ref{iaai}, $W^2(\bl)$ is a unitary $G \times T$-representation, and
the study of $W^2(\bl)$ can be replaced by $W^2(e^{-F})$.
Let $\la \in C$, and we define $W_\la$ in (\ref{fupi}).
Let $0 \neq f = \sum_I f_I \xi^I \in W_\la$.
We have
\[
\begin{array}{rll}
f \in W^2(e^{-F})
& \Longleftrightarrow f_I \in \ch^2(G_\0 A, e^{-F}) \mbox{ for all }
I \in S^c
& \mbox{by (\ref{adam})} \\
& \Longleftrightarrow \la \in \im(\frac{1}{2}F')
& \mbox{by \cite[Thm.2]{jfa}} \\
& \Longleftrightarrow \la \in \im(\Phi)
& \mbox{by Theorem \ref{thm1}}.
\end{array}
\]
By \cite[Thm.1]{cfv}, $W_\la \cong \Theta_{\la+ \rho}$.
It follows that $W^2(\bl) \cong W^2(e^{-F})
\cong \sum_{\rm Im (\Phi)} \Theta_{\la + \rho}$.$\hfill$$\Box$


\newpage
\section{Cells}
\setcounter{equation}{0}

In this section, we prove Theorems \ref{thm3} and \ref{thm4}.
Recall that $\fl$ is a contragredient Lie superalgebra
with real form $\fg$, and it has an admissible positive system (\ref{howa})
so that $C \neq \emptyset$ in (\ref{vgood}).
Let $\Pi_c$ be the compact simple roots.
We have a correspondence between the subsets
$R \subset \Pi_c$ and the cells $\si \subset C$
defined in (\ref{sel}).
In this way,
$C$ is a disjoint union of all the cells.

Fix a cell $\si$, and equivalently $R \subset \Pi_c$.
For any root $\al$, let $\ker (\al) \subset \fh$ denote its kernel.
Let
\[ \fh_\si = \cap_R \ker(\al), \]
and by convention $\fh_\si = \fh$ if $R = \emptyset$.
Let $\fn = \sum_{\De^+} \fl_\al$, and $\fb = \fh + \fn$
is a Borel subalgebra.
Let $\overline{R} \subset \De^+$
be the positive roots that are non-negative linear combinations
of elements of $R$.
It determines a parabolic subalgebra
$\fp = \fb + \sum_{- \overline{R}} \fl_\al$.
Let $[\fp, \fp]$ be its commutator subalgebra.
Let $\fh_R = \fh \cap [\fp,\fp]$, and we have
\[ \fp = \fh + \sum_{\De^+ \cup (- \overline{R})} \fl_\al , \qquad
[\fp, \fp] = \fh_R + \sum_{\De^+ \cup (- \overline{R})}\fl_\al .\]

We have a direct sum $\fh = \fh_\si + \fh_R$.
By intersecting with $\ft$ and $\fa$,
we obtain
\beq
 \ft = \ft_\si + \ft_R , \qquad \fa = \fa_\si + \fa_R . \label{eapp}
 \eeq

Let $\fg^\si = \fg \cap \fp$, and it is the centralizer of $\si$ in $\fg$
under the coadjoint action.
Its commutator subalgebra
$\fg_{\rm ss}^\si = \fg \cap [\fp,\fp]$
is a semisimple ordinary Lie algebra (see proof of Proposition \ref{ranc}).
By the Iwasawa decomposition,
$\fl=\fg \oplus \fa \oplus \fn$.
We have $\fa_R , \fn \subset [\fp,\fp]$, so
\beq\label{lie-id}
\fl/[\fp,\fp] \cong
\fg/(\fg \cap [\fp,\fp]) \oplus \fa_\si
= \fg/\gss \oplus \fa_\sigma .
\eeq

The Lie algebra $\gss$ corresponds to Lie subgroup $G_{\rm ss}^\si \subset G_\0$.
We are interested in the space
\[
 X_\si = G/G_{\rm ss}^\si \times A_\si \;\;,\;\;
 \xs = G_\0/G_{\rm ss}^\si \times A_\si .
 \]
By (\ref{lie-id}), the map
\beq
X_\si=G/G_{\rm ss}^\si \times A_\si  \hookrightarrow L/(P,P) .
\label{komp}
\eeq
is a super diffeomorphism at each point as the tangent spaces
are isomorphic, and it is an ordinary diffeomorphism of
$\xs$ onto its open image in
$(L/(P,P))_{\bar{0}}$.
Hence $X_\si$ is diffeomorphic to an open subsupermanifold in $L/(P,P)$.
In this way, $X_\si$ inherits a natural
complex structure.

In the special case where $\si$ is the interior of $C$, we have
$R=\emptyset$, $\fh_\si = \fh$, $P =B$, $(P,P) = N$,
$G^\si = T$ and $G_{\rm ss}^\si = \{e\}$, so $X_\si$ is just $GA$
discussed previously.

We have a $G \times H_\si$-action on $X_\si$, as
$T_\si$ normalizes $G_{\rm ss}^\si$.
Similar to (\ref{motiva}), (\ref{mox}) and (\ref{niti}), and using
the complex structure inherited from (\ref{komp}), we have
\beq \Om^{p,q}(X_\si)^{G \times T_\si}
= C^\infty (A_\si) \otimes \wedge_{\ft_\si}^{p,q}(\fg + \fas,  \gss)^* ,
\label{emyy}
\eeq
where
\[ \begin{array}{l}
\wedge_{\ft_\si}^{p,q}(\fg + \fas,  \gss)^* =
\{\om \in \we^{p,q} (\fg + \fas)^* \;;\; \\
\ad_u^* \om  = 0 \mbox{ for all } u \in \ft_\si
\mbox{ and } \ad_v^* \om = \imath(v) \om =0 \mbox{ for all } v \in \fg_{\rm ss}^\si \} .
\end{array}
\]
The Dolbeault operators $\partial$ and $\bar{\partial}$ raise the degrees of
$p$ and $q$ respectively.

By (\ref{eapp}), we have the Cartesian product $A = A_\si A_R$, so we can identify
$C^\infty(A_\si)$ with the $A_R$-invariant elements of $C^\infty(A)$. In this way
we have the subspace
\begin{equation}
 C^\infty (A_\si) \otimes \wedge_{\ft_\si}^{p,q}(\fg + \fas,  \gss)^*
\subset C^\infty(A) \otimes \wedge_\ft^{p,q}(\fg+\fa)^* ,
\label{thth}
\end{equation}
where the latter is discussed in Section 3.
The natural fibration $L/N \lra L/(P,P)$ leads to the
inclusion map in (\ref{thth}).

\sk

\noindent {\it Proof of Theorem \ref{thm3}:}

Let $\om \in \Om^{1,1}(X_\si)^{G \times T_\si}$ be a closed element.
By (\ref{emyy}) and (\ref{thth}),
we treat it as an element of $C^\infty(A) \otimes \wedge_\ft^{p,q}(\fg+\fa)^*$.
So by Theorem \ref{thm1}, we obtain
$\om = i \partial \bar{\partial}F$ where $F \in C^\infty(A)$.
But since $\om$ belongs to the subspace in (\ref{thth}), we have
$F \in C^\infty(A_\si)$.
We identify it with $\tf \in C^\infty(\fa_\si)$ by Remark \ref{ide}.

Let $x_1,...,x_p$ and $y_1,...,y_p$ be the linear coordinates on
$\fa_\si$ and $\ft_\si$ respectively.
Let
\[ \fg / \gss + \fa_\si = \fh_\si + V_\si \;\;,\;\;
\mbox{where } V_\si = \fg \cap \sum_{\De^+ \bsl \overline{R}} (\fl_\al + \fl_{-\al}) .\]
Similar to (\ref{bntt}), $\om_a$ is nondegenerate
if and only if
\beq
\om|_{\fh_\si} =
\frac{1}{2} \sum_{i,k} \frac{\partial^2 \tf}{\partial x_i \partial x_k} dx_j\we \xi_k
\;\;,\;\;
\om|_{V_\si} = \frac{1}{2} \sum_k \frac{\partial \tf}{\partial x_k} d \xi_k
\label{shiq}
\eeq
are both nondegenerate.
We follow the proof of Theorem \ref{thm1}.
The first term of (\ref{shiq}) is nondegenerate
if and only if the Hessian matrix of $\tf$
is nondegenerate
everywhere.
The second term of (\ref{shiq}) is nondegenerate
if and only if $F'(a) \in (\fa_\si^*)_{\rm reg}$.
Similar to (\ref{sima}),
the moment map is $\Phi_\si(ga) = \frac{1}{2} F'(a)$
 (see notation
via the functor of points in Remark \ref{fopts} and also \cite[3.2]{ccf}).
This proves Theorem \ref{thm3}.
$\hfill$$\Box$

\sk

Let $\om = i \partial \bar{\partial}F$ be given by Theorem 1.3,
and suppose that $F$ is strictly convex.
Let $\om_\zero$ be its restriction to
$(X_\si)_\0 = (G_\0/G_{\rm ss}^\si)A_\si$.
Let $\blz^\sigma$ be the complex line bundle on $(X_\si)_\zero$,
whose Chern class is the cohomology class $[\omz] =0$.
It has an invariant Hermitian structure $\langle \cdot , \cdot \rangle_{\bar{0}}$,
and there exists a nowhere vanishing
$G_\zero \times T_\si$-invariant holomorphic section
$u_\zero^\si$ such that
\[
\langle u_\zero^\si , u_\zero^\si \rangle_{\bar{0}} = e^{-F} .
\]
See \cite[Prop.4.2]{jfa}. We let $u_\zero^\si$ denote this specific section,
and later $u^\si$ its extension.

The line bundle $\blz^\si$ extends
to the $1|0$ vector bundle $\bl^\si$ on $X_\si$, and
$u_\zero^\si$ extends uniquely to
$u^\si \in \cH(\bl^\si)$, similarly to the corresponding result in Section 4
regarding $\blz$, $\bl$ and $u$.
Since $u_\zero^\si$ is nowhere vanishing, so is $u^\si$. We have an identification
between $\cH(\bl^\si)$ and the free ${1|0}$ module $\cH(X_\si)$
of holomorphic functions on $X_\si$:
\beq
\cH(X_\si)  \lra  \cH(\bl^\si) \;,\;
f  \mapsto  f u^\si .
\label{uumn}
\eeq

We define, both on $\cH(X_\si)$ and $\cH(\bl^\si)$, an action
of $G \times T_\si$ via SHCP, similarly to what we did in Section 4
for $G \times T$. The action of $G$
on $\cH(X_\si)$ is via left invariant vector fields (see \cite[Sec.8.3]{ccf}).
On $\cH(\bl^\si)$ we have:
$$
g_\zero \cdot fu^\si :=
(g_\zero \cdot f)u^\si , \qquad
X \cdot f u^\si =
(X \cdot f)u^\si , \qquad g \in G_\zero, \, X\in \fg .
$$
We also have a right $T_\sigma$-action $R$
 on  $\cH(X_\si)$ given by the dual of
right translation. On $\cH(\bl^\si)$, we have
$R_t (fu^\si) = (R_t f)u^\si$ for all
$t \in T_\si$ and $f \in \cH(X_\si)$.
In this way, $u^\si$ is $G \times T_\si$-invariant. The $G$ and $T_\si$
actions commute.

For $\la \in \si$, we define
$W(\bl^\si)_\la$ (resp. $W(X_\si)_\la$)
to be the irreducible submodule of $\ch(\bl^\si)_\la$ (resp. $\ch(X_\si)_\la$)
which contains the highest weight vector.

\begin{proposition}
Let $\la \in \si$. \label{ranc}

\noindent{\rm (a)} \, $W(\bl^\si)_\la \cong W(X_\si)_\la \cong \Theta_{\la + \rho}$,

\noindent{\rm (b)} \, $X_\si$ has a $G \times H_\si$-invariant measure of the form
$\mu_{X_\si} = \mu_{\xs} \mu_{G_\1}$,
where $\mu_{\xs}$ is a
$G_\0 \times H_\si$-invariant measure of $\xs$,
and $\mu_{G_\1}$ is a $G$-invariant measure
of $G_\1 = G/G_\0$.
\end{proposition}
\begin{proof}
The trivialization (\ref{uumn})
implies that $W(\bl^\si)_\la \cong W(X_\si)_\la$.
Next we show that $W(X_\si)_\la \cong \Theta_{\la + \rho}$.
The fibration $\pi : X \lra X_\si$ is $G \times T_\si$-equivariant,
so it leads to an injection of $G$-modules,
$\pi^* : \cH(X_\si)_\la \hookrightarrow \cH(X)_\la$.
By \cite[Thm.2.12,4.27]{cfv},
we have $W(X)_\la \cong \Theta_{\la + \rho}$,  and it is irreducible.
Hence its $G$-submodule $\pi^*(W(X_\si)_\la)$ is either 0 or $\Theta_{\la + \rho}$.
In the ordinary setting, $\cH((X_\si)_\zero)_\la$ is a discrete series representation of
$G_\zero$, so it is not 0 \cite{hc-v}. It follows that
$W(X_\si)_\la \cong \Theta_{\la + \rho}$.
This proves part (a).

For part (b), we first recall the root space decomposition
$\fg = \ft + \sum_{\De^+} \fg_\al$ in (\ref{haba}).
By (\ref{hccone}) and the discussions preceding (\ref{lie-id}),
\beq
\fg^\si = \ft + \sum_{\overline{R}} \fg_\al \;,\;
\fg_{\rm ss}^\si = \ft_R + \sum_{\overline{R}} \fg_\al ,
\label{dexx}
\eeq
where $\overline{R}$ consists of the positive roots $\al$ such that
$\la(h_\al) = 0$ for all $\la \in C$.
All $\la \in C$ and $\be \in \De_\1^+$ satisfy $\la(h_\be) < 0$,
so $\overline{R}$ does not contain any odd root.
Hence (\ref{dexx}) implies that:
$\fg^\si , \fg_{\rm ss}^\si \subset \fg_\0$,
and $\fg^\si$ is reductive, and $\fg_{\rm ss}^\si$ is semisimple.
 By Example \ref{unim}, semisimple and reductive Lie groups are
 unimodular, so
 \beq
G_{\rm ss}^\si \mbox{ and } G^\si \mbox{ are unimodular Lie groups.}
\label{gunak}
 \eeq
We apply Proposition \ref{ffub} with $U = G_{\rm ss}^\si$
and obtain a $G$-invariant measure of $G/G_{\rm ss}^\si$ given by
$\mu_{G/G_{\rm ss}^\si} = \mu_{G_\0/G_{\rm ss}^\si} \mu_{G_\1}$,
where $\mu_{G_\0/G_{\rm ss}^\si}$ is a $G_\0$-invariant measure of
$G_\0/G_{\rm ss}^\si$, and $\mu_{G_\1}$ is a $G$-invariant measure
of $G_\1 = G/G_\0$.
Since $T_\si$ is compact,
$\mu_{G_\0/G_{\rm ss}^\si}$ is automatically $G_\0 \times T_\si$-invariant.
We take the product of $\mu_{G/G_{\rm ss}^\si}$ with the Haar measure
of $A_\si$ and obtain the desired result of part (b).
 \end{proof}

Consider $\oplus_\si W(\bl^\si)_\la$ and $\oplus_\si W(X_\si)_\la$,
summed over all integral weights $\la \in \si$.
Let
\beq\label{hatsi}
( s,t ) :=
\int_{X_\si} \langle s,t \rangle \, \mu_{X_\si}  \;,\;
s,t \in \oplus_\si W(\bl^\si)_\la ,
\eeq
and
\beq
(f,h) = \int_{X_\si} f \bar{h} e^{-F} \, \mu_{X_\si} \;,\; f,h \in \oplus_\si W(X_\si)_\la .
\label{hemik}
\eeq
We shall show that (\ref{hatsi}) and (\ref{hemik}) provide
positive definite Hermitian metrics on the square integrable elements,
so we can define their completions
$W^2(\bl^\si)$ and $W^2(X_\si , e^{-F})$ respectively.

\begin{proposition}
 $W^2(\bl^\si)$ and $W^2(X_\si , e^{-F})$ are equivalent
 unitary $G \times T_\si$-modules with respect to
 the positive definite super Hermitian metrics (\ref{hatsi})
 and (\ref{hemik}) respectively.
 \label{nchu}
\end{proposition}
\begin{proof}
Similar to Proposition \ref{iaai}, the square integrable elements of
(\ref{hatsi}) and (\ref{hemik}) are identified by
$f u^\si \mapsto f$, where $u^\si$ is the section in (\ref{uumn}).
It remains to show that $W^2(X_\si ,e^{-F})$ is a unitary $G \times T_\si$-representation
with respect to (\ref{hemik}). Let
\[ (\oplus_\si W(X_\si)_\la)^2 =
\{f \in \oplus_\si W(X_\si)_\la \;;\; (f,f)
\mbox{ converges}\} .\]
The arguments resemble Proposition \ref{veca}, namely we check
the following four steps:
\beq
 \begin{array}{cl}
\mbox{(a)} & (\oplus_\si W(X_\si)_\la)^2 \mbox{ is a super vector space}, \\
\mbox{(b)} & \mbox{the $G \times T_\si$-action on $W^2(X_\si ,e^{-F})$ preserves (\ref{hemik})},\\
\mbox{(c)} & \mbox{for any $G \times T_\si$-invariant super Hermitian form $H$ on
$W(X_\si)_\la \oplus W(X_\si)_\nu$}, \\
& \mbox{$H(W(X_\si)_\la,W(X_\si)_\nu) = H(W(X_\si)_\la^a, W(X_\si)_\la^b) = 0$
for $\la \neq \nu$ and $a \neq b$,} \\
\mbox{(d)} & \mbox{the $G$-invariant Hermitian forms on each $W(X_\si)_\la$ are unique
up to scalar.}
\end{array}
\label{pdv}
\eeq
Since the arguments are similar to (\ref{pdu}), we merely sketch the ideas here.

Let $P$ be the power set of $\{1,...,\dim \fg_\1\}$.
By Proposition \ref{ranc}(b), we have
\beq
 \mu_{X_\si} = \mu_{\xs} \mu_{G_\1}
= \mu_{\xs} (\sum_{K \in S} a_K \xi^K \bar{\xi}^K) \label{tyaa}
\eeq
for some $S \subset P$ and non-zero constants $\{a_K\}_{K \in S}$.
Let $f = \sum_{I \in P} f_I \xi^I \in \oplus_\si W(X_\si)_\la$.
Similar to (\ref{sxff}) and (\ref{leng}), we have
\beq
(f,f)
 = \int_{G_\1} \int_{\xs} f \bar{f} e^{-F} \,
\mu_{\xs} \, \mu_{G_\1}
= \sum_{I \in S^c} c_I \int_{\xs} f_I \bar{f}_I e^{-F} \,
\mu_{\xs} ,
\label{attg}
\eeq
where $c_I = \int_{G_\1} a_{I^c} \xi^I \bar{\xi}^I \xi^{I^c} \bar{\xi}^{I^c} \neq 0$
for all $I \in S^c$.
It follows that $f \in (\oplus_\si W(X_\si)_\la)^2$
if and only if $f_I \in \ch^2 (\xs, e^{-F})$
for all $I \in S^c$. Since the latter is a vector space,
so is the former. This proves (\ref{pdv})(a).

Similar to the arguments of
(\ref{agek}), the $G$-action preserves (\ref{hemik}).
Since $\mu_{G_\0/G_{\rm ss}^\si}$ is right $T_\si$-invariant,
the right $T_\si$-action also preserves (\ref{hemik}). This proves (\ref{pdv})(b).
The proofs of (\ref{pdv})(c,d) are identical to (\ref{pdu})(c,d).
This proves (\ref{pdv}).

Let $H$ be the Hermitian form (\ref{hemik}).
By (\ref{pdv})(b), the $G$-action preserves $H$,
so $(\oplus_\si W(X_\si)_\la)^2 = \oplus_\tau W(X_\si)_\la$
for some $\tau \subset \si$.
Let $\la \in \tau$.
By Proposition \ref{eeej} and (\ref{pdv})(d),
there exists $c \in \bc$ such that $cH$
is positive definite on the unitarizable $\fg$-module $W(X_\si)_\la$.
But by (\ref{attg}), we have $H(f,f) > 0$
for all ordinary non-zero functions $f$,
so $c>0$, and $H$ is already positive definite on $W(X_\si)_\la$.
By (\ref{pdv})(c), distinct summands of $\oplus_\tau W(X_\si)_\la$
are orthogonal,
so $H$ is positive definite on  $\oplus_\tau W(X_\si)_\la$.
We take its completion with respect to $H$,
and obtain the unitary representation $W^2(X_\si ,e^{-F})$.
\end{proof}

\sk

\noindent {\it Proof of Theorem \ref{thm4}:}

By Proposition \ref{nchu}, we may consider $W^2(X_\si, e^{-F})$ in place of $W^2(\bl^\si)$.
Let $S$ be the index set in (\ref{tyaa}). Let $\la \in \si$.
Given $0 \neq f = \sum_I f_I \xi^I \in W(X_\si)_\la$, we have
\[
\begin{array}{rll}
f \in W^2(X_\si, e^{-F}) & \Longleftrightarrow
f_I \in \ch^2((X_\si)_\0 , e^{-F}) \mbox{ for all } I \in S^c
& \mbox{by (\ref{attg})} \\
& \Longleftrightarrow \la \in \mbox{Im}(\frac{1}{2} F')
& \mbox{by \cite[Thm.2]{jfa}} \\
& \Longleftrightarrow \la \in \mbox{Im}(\Phi_\si)
& \mbox{by Theorem \ref{thm3}.}
\end{array}
\]
By Proposition \ref{ranc}(a), $W(X_\si)_\la \cong \Theta_{\la + \rho}$, so
$W^2(\bl^\si) \cong W^2(X_\si, e^{-F}) \cong \sum_{\im(\Phi_\si)} \Theta_{\la + \rho}$.
This proves Theorem \ref{thm4}.$\hfill$$\Box$


\newpage
\section{Models}
\setcounter{equation}{0}

The notion of a model was first proposed by Gelfand for
compact Lie groups \cite{gz}.
It refers to a unitary representation on a Hilbert space
in which every irreducible representation occurs once.
It has been extended to non-compact semisimple Lie groups
\cite{jfa}, and here we construct models of
real forms $G$ of contragredient Lie supergroups.
A model of highest weight Harish-Chandra supermodules
is a unitary $G$-representation on a super Hilbert space
in which every irreducible highest weight Harish-Chandra supermodule
occurs once.
In this section, we construct such a model and prove Theorem \ref{thm5}.

Fix a cell $\si \subset \fa^*$, see (\ref{sel}).
Let $\si = \{\sum_j c_j \la_j \;;\; c_j > 0\}$ for some $\la_j \in \fa^*$.
In (\ref{bebas}), we define $\tf_\si \in C^\infty(\fa_\si)$ by
$\tf_\si (x) = \sum_j \exp (\la_j(x))$.
We have
\beq d \tf_\si : \fa_\si \lra \fa_\si^* \;\;,\;\;
d \tf_\si(x) = \sum_j \exp (\la_j(x)) \la_j ,\label{awat}
\eeq
and its Hessian matrix is the diagonal matrix with positive entries
$\exp (\la_j(x))$.
We identify $\tf_\si$ with $F_\si \in C^\infty(A_\si)$ by Remark \ref{ide}.
Then $F_\si$ is strictly convex, and by (\ref{awat}),
$\mbox{Im}(F_\si') =
\si \subset (\fa_\si^*)_{\rm reg}$.
By Theorem \ref{thm3}, $\om_\si = i \partial \bar{\partial}F_\si$
is a $G \times T_\si$-invariant pseudo-K\"ahler form
on $X_\si$.
Its moment map has image $\im(\Phi_\si) = \im(\frac{1}{2}F_\si') = \si$.
So by Theorem \ref{thm4},
\beq
W^2(\bl^\si) \cong \sum_{\la \in \si} \Theta_{\la + \rho} ,
\label{bane}
\eeq
summed over the integral weights $\la$.
We repeat this construction for each cell $\si$, and take the sum
over the collection $\{\si\}$ of all the cells. By (\ref{bane}),
\beq \sum_{\{\si\}} W^2(\bl^\si) =
\sum_{\{\si\}} \sum_{\la \in \si} \Theta_{\la + \rho} =
\sum_{\la \in C} \Theta_{\la + \rho} .
\label{saat}
\eeq
The irreducible highest weight Harish-Chandra supermodules
$\Theta_{\la + \rho}$ are parametrized by the integral weights
$\la \in C$ of (\ref{vgood}). So (\ref{saat}) is a
model for such supermodules.
This proves Theorem \ref{thm5}.


\newpage
\section{Symplectic reduction}
\setcounter{equation}{0}

In this section, we discuss symplectic reduction
and prove Theorems \ref{thm6} and \ref{thm7}.
Let $X_\si = G/G_{\rm ss}^\si \times A_\si$ be equipped with
a $G \times T_\si$-invariant
pseudo-K\"ahler form $\om = i \partial \bar{\partial}F$,
with moment map $\Phi_\si : X_\si \lra i\ft_\si^*$ of the right $T_\si$-action.
Let $\lat \in \mbox{Im}(\Phi_\si) \subset (i\ft_\si^*)_{\rm reg}$.
By Theorem \ref{thm3},
the Hessian matrix is nondegenerate,
so the gradient map $F'$ is a local diffeomorphism.
We have $\Phi_\si(ga) = \frac{1}{2} F'(a)$,
so the restriction of $\Phi_\si$ to $A_\si \subset X_\si$
is also a local diffeomorphism. Hence there exists a
discrete set $\Gamma \subset A_\si$ such that
\beq
\iimath : \Phi_\si^{-1}(\lat) = G/G_{\rm ss}^\si \times \Gamma
\hookrightarrow X_\si ,
\label{cbsp}
\eeq
where $\iimath$ is the inclusion.
There is a right $T_\si$-action on $\Phi_\si^{-1}(\lat)$ because $T_\si$
normalizes $G_{\rm ss}^\si$. This leads to the quotient map
\beq
 \pi : \Phi_\si^{-1}(\lat) \lra \Phi_\si^{-1}(\lat)/T_\si = G/G^\si \times \Gamma .
 \label{sjij}
 \eeq

In \cite[\S7]{jfa}, we show that if $G$ is an ordinary Lie group,
there exists a unique symplectic form $\om_{\lat}$ on
$G/G^\si \times \Gamma$ such that
\beq
\pi^* \om_{\lat} = \iimath^* \om .
\label{syre}
\eeq
Write $(X_\si)_{\lat} = G/G^\si \times \Gamma$.
The process $(X_\si, \om) \leadsto ((X_\si)_{\lat}, \om_{\lat})$
is called symplectic reduction \cite{mw}.
We shall show that (\ref{syre}) holds in our super setting,
then compute $\om_\la$ and show that it is pseudo-K\"ahler.

\sk

\noindent {\it Proof of Theorem \ref{thm6}:}

Let $u \in \ft$ and $\xi \in \ft^* \subset \fg^*$.
Here $\ft^* \subset \fg^*$ is given by annihilation of the root spaces.
We claim that
\beq
\mbox{(a)} \, \imath(u) (d \xi) = 0
\;\;,\;\; \mbox{(b)} \, \ad_u^* (d \xi) =0 .
\label{harap}
\eeq
In (\ref{harap}), we regard
$d \xi \in \we^2 \fg^*$ and $\ad_u^* \in {\rm End}(\we^n \fg^*)$
(not $\we^2 \ft^*$ and ${\rm End}(\we^n \ft^*)$).
Recall from (\ref{haba}) that $\fg = \ft + V$,
where $V$ is the sum of root spaces.
Given $x+y \in \ft + V$, we have
$(d \xi)(u, x+y) = \xi([u,x] + [u,y]) = 0$ because $[u,x] = 0$ and $[u,y] \in V$.
Hence $\imath(u) (d \xi) =0$, which proves (\ref{harap})(a).

 To prove (\ref{harap})(b), let $x,y \in \fg$ and we have
\beq \begin{array}{l}
 (\ad_u^*(d \xi))(x,y)  = \xi ([[u,x],y] + [x,[u,y]])
= \xi ([u,[x,y]]) \\
= (d \xi)(u,[x,y]) = (\imath(u)(d \xi))([x,y]) =0 . \label{ahir}
\end{array}
\eeq
The last expression of (\ref{ahir}) vanishes because
$\imath(u)(d \xi) =0$ by (\ref{harap})(a). This proves (\ref{harap})(b).

By Remark \ref{ide}, we identify $F$ with $\tf \in C^\infty(\fa_\si)$.
Recall from (\ref{shiq}) that $\om = \om|_{\fh_\si} + \om|_{V_\si}$, where
\beq
\om|_{\fh_\si} = \frac{1}{2}
\sum_{i,k} \frac{\partial^2 \tf}{\partial x_i \partial x_k} dx_j\we \xi_k
\;,\;
\om|_{V_\si} = \frac{1}{2}  \sum_k \frac{\partial \tf}{\partial x_k} d \xi_k .
\label{faww}
\eeq
Here $x_j$ are the linear coordinates on $\fa_\si$, and $\xi_k \in \ft_\si^*$.
The expression $dx_j\we \xi_k$ means that $\ft_\si$ and $\fa_\si$ are
Lagrangian subspaces of $\om|_{\fh_\si}$.
The image of $\iimath$ of (\ref{cbsp}) intersects the submanifold $A_\si \subset X_\si$
on the discrete set $\Gamma$, so $\iimath^* (\om|_{\fh_\si}) = 0$.
It follows that
\beq
 \iimath^* \om = \frac{1}{2} \sum_k \frac{\partial \tf}{\partial x_k} d \xi_k .
 \label{tati}
 \eeq
By (\ref{harap}) and (\ref{tati}),
\beq \imath(u) (\iimath^* \om) = \ad_u^* (\iimath^* \om) =0 \;\; \mbox{ for all } u \in \ft .
\label{kkc}
\eeq

Since $\iimath$ is $G$-equivariant, $\iimath^* \om$ is $G$-invariant.
Hence $\iimath^* \om \in \Om^2(G/G_{\rm ss}^\si)^G = \we^2(\fg, \fg_{\rm ss}^\si)^*$
(see definition in (\ref{motiva})). This, together with (\ref{kkc}), imply that
$\iimath^* \om \in \we^2(\fg, \fg^\si)^*$. So there exists
$\om_{\lat} \in \Om^2(G/G^\si)^G$ such that
$\pi^* \om_{\lat} = \iimath^* \om$, where $\pi$ is the map (\ref{sjij}).
Here $\om_\la$ is unique because $\pi^*$ is injective.
This proves (\ref{syre}).

Next we compute $\om_{\lat}$.
By (\ref{syre}) and (\ref{tati}),
$\om_{\lat} =
\frac{1}{2} \sum_k \frac{\partial \tf}{\partial x_k}(a) d \xi_k$.
By Theorem \ref{thm3}, $\lat = \Phi_\si(a) = \frac{i}{2} \sum_k \frac{\pa \tf}{\pa x_k}(a) \xi_k$, so
\beq
 -i d \lat = \frac{1}{2} \sum_k \frac{\pa \tf}{\pa x_k}(a) d \xi_k = \om_{\lat} .
 \label{baz}
 \eeq

Finally we show that $\om_\la$ is pseudo-K\"ahler.
The supermanifolds $(G/G_{\rm ss}^\si)A_\si$ and $G/G^\si$
acquire their complex structures as
open subsets of $L/(P,P)$ and $L/P$ respectively.
The quotient map $L/(P,P) \lra L/P$ leads to a holomorphic fibration
\[ \rho : G/G_{\rm ss}^\si \times A_\si \lra G/G^\si .\]
By (\ref{faww}) and (\ref{baz}),
$\rho^* \om_\la = \om|_{V_\si}$.
Since $\om$ is a $(1,1)$-form and $\om(\fh_\si,V_\si) = 0$, it follows that
$\om|_{V_\si}$ is also a $(1,1)$-form, and so is $\om_\la$.
Since $\om|_{V_\si}$ is nondegenerate, so is $\om_\la$.
Hence $\om_\la$ is pseudo-K\"ahler.
This proves Theorem \ref{thm6}.$\hfill$$\Box$

\sk

Guillemin and Sternberg \cite{gs2} propose that
symplectic reduction is the geometric analogue of taking subrepresentation.
This is known as ``quantization commutes with reduction''.
We now prove Theorem \ref{thm7}, which shows that our setting fulfills this principle.

\sk

\noindent {\it Proof of Theorem \ref{thm7}:}

By (\ref{gunak}), $G^\si$ is a unimodular Lie group (see Definition \ref{whted}).
Since $G_\0$ is reductive \cite[\S2]{ka}, by Example \ref{unim}, it is also unimodular.
So $G_\0/G^\si$ has a
$G_\0$-invariant measure $\mu_{G_\0/G^\si}$ \cite[Prop.8.36]{kn}.
We recall Harish-Chandra's construction of the discrete series representations.
Fix an integral weight $\la \in \si$. It determines a homogeneous line bundle
$\bl_\0^\la = G_\0 \times_\la \bc$ over $G_\0/G^\si$. There exists a Hermitian structure
$\langle \cdot , \cdot \rangle$ on the holomorphic sections $\ch(\bl_\0^\la)$,
and we define
\[ \ch^2(\bl_\0^\la) = \{s \in \ch(\bl_\0^\la) \;;\;
\int_{G_\0/G^\si} \langle s , s \rangle \, \mu_{G_\0/G^\si} < \infty\} .\]
Then $\ch^2(\bl_\0^\la)$ is irreducible, and is
the discrete series representation with Harish-Chandra parameter $\la + \rho$ \cite{hc-vi}.

We extend Harish-Chandra's construction to the super setting.
Since $G$ is a Lie supergroup and $G^\si$ is an ordinary Lie group,
$G/G^\si$ is globally split, hence
$\cH(G/G^\si)\cong \cH(G_\0/G^\si) \otimes \wedge(\xi)$. If
we fix such an isomorphism,
the line bundle $\bl_\0^\la$ on $G_\0/G^\si$ extends to the $1|0$ vector bundle
$\bl^\la$ on $G/G^\si$ (see \cite[Ch.4]{ccf}, \cite[Ch.6]{kessler}),
and we can view
$\cH(\bl_\0^\la)$ as a subalgebra of $\cH(\bl^\la)$,
the holomorphic sections on $\bl^\la$.
We extend the Hermitian structure to $\cH(\bl^\la)$ by
$\langle s \xi_1 , t \xi_2 \rangle =
\langle s , t  \rangle \xi_1 \bar{\xi}_2$.

Let $W(\bl^\la)$ be the irreducible $G$-submodule of $\ch(\bl^\la)$
which contains the highest weight vector.
Since $G^\si$ is unimodular, we can apply Proposition \ref{ffub} with $U = G^\si$,
namely there exists a $G$-invariant measure of $G/G^\si$ given by
\[ \mu_{G/G^\si} = \mu_{G_\0/G^\si} \mu_{G_\1}
= \mu_{G_\0/G^\si} (\sum_{K \in S} a_K \xi^K \bar{\xi}^K) ,\]
where $S$ is a subset of the power set of $\{1,..., \dim \fg_\1\}$,
and $a_K \neq 0$ for all $K \in S$.
We then define
\beq
W^2(\bl^\la) = \{s \in W(\bl^\la) \;;\;
\int_{G/G^\si} \langle s,s \rangle \, \mu_{G/G^\si} \mbox{ converges}\} .
\label{angi}
\eeq
Let $s = \sum_I s_I \xi^I \in W(\bl^\la)$. Similar to (\ref{attg}), we have
\[ \int_{G/G^\si} \langle s,s \rangle \, \mu_{G/G^\si}
= \sum_{I \in S^c} c_I \int_{G_\0/G^\si} \langle s_I, s_I \rangle \, \mu_{G_\0/G^\si} ,\]
where $c_I \neq 0$ for all $I \in S^c$.
Thus $s \in W^2(\bl^\la)$ if and only if $s_I \in \ch^2(\bl_\0^\la)$ for all $I \in S^c$.
Here $\ch^2(\bl_\0^\la)$ is a discrete series representation \cite{hc-vi},
so $W^2(\bl^\la)$ is a non-trivial vector space.
The Hermitian form (\ref{angi}) is $G$-invariant, so
$W^2(\bl^\la)$ is a $G$-subrepresentation of $W(\bl^\la)$.
But by \cite[Thm.4.2.7]{cfv},
$W(\bl^\la)$ is irreducible and is given by $\Theta_{\la + \rho}$.
Hence as algebraic $G$-representations,
\beq W^2(\bl^\la) \cong \Theta_{\la + \rho} .
\label{bilo}
\eeq
Similar to the arguments of Proposition \ref{nchu},
$W^2(\bl^\la)$ is a unitary $G$-representation with respect to (\ref{angi}).
The holomorphic sections form a closed subspace within the
square integrable sections,
so $W^2(\bl^\la)$ is complete.

By Theorem \ref{thm6}, the symplectic quotient is
$(G/G^\si \times \Gamma , -id \la)$.
If $F$ is strictly convex, then $F'$ is injective,
so there exists a unique $a \in A_\si$ such that
$\frac{1}{2}F'(a) = \lat$,
namely $\Gamma = \{a\}$.
We write the
symplectic quotient as
\[ ((X_\si)_\la, \om_\la) =
(G/G^\si , -id \la) .\]
We perform geometric quantization to it.
By (\ref{angi}) and (\ref{bilo}), we write
\beq \QQ((X_\si)_{\lat} , \om_{\lat}) = W^2(\bl^\la) \cong \Theta_{\la + \rho} ,
\label{jadi}
\eeq
where $\QQ$ denotes geometric quantization.
 By Theorem \ref{thm4},
 $\QQ(X_\si, \om)_{\lat} = W^2(\bl^\si)_{\lat} \cong \Theta_{\la + \rho}$.
So together with (\ref{jadi}), we have
\[  \QQ((X_\si)_{\lat} , \om_{\lat}) \cong \QQ(X_\si, \om)_{\lat} .\]
This proves Theorem \ref{thm7}.$\hfill$$\Box$



\newpage

\begin{thebibliography}{99}



\bibitem{all2} A. Alldridge and J. Hilgert,
\textit{Invariant Berezin integration on homogeneous supermanifolds},
J. Lie Theory {\bf 20} no. 1 (2010), 65-91.




\bibitem{ccf} C. Carmeli, L. Caston and R. Fioresi,
with an appendix by I.~Dimitrov,
{\it  Mathematical Foundation of Supersymmetry},
EMS Ser.~Lect.~Math., European Math.~Soc., Zurich, 2011.

\bibitem{cfv}
C. Carmeli, R. Fioresi and V. S. Varadarajan,
\textit{Highest weight Harish-Chandra supermodules and their
geometric realizations}, Transf. Groups
{\bf 25}, no. 1 (2020), 33-80.

\bibitem{cfv2}
C. Carmeli, R. Fioresi and V. S. Varadarajan,
\textit{Unitary Harish-Chandra representations of real supergroups},
J. Noncommut. Geom. {\bf 17} (2023), 287-303.




\bibitem{cw}
S. J. Cheng and W. Wang,
{\it Dualities and Representations of Lie Superalgebras,}
Grad. Studies in Math. {\bf 144}, Amer. Math. Soc., 2012.




\bibitem{jfa}
M. K. Chuah,
 \textit{Holomorphic discrete models of semisimple Lie groups
and their symplectic constructions},
J. Funct. Anal. {\bf 175} (2000), 17-51.



\bibitem{cf}
M. K. Chuah and R. Fioresi,
 \textit{Hermitian real forms of contragredient Lie superalgebras},
J. Algebra {\bf 437} (2015), 161-176.



\bibitem{coulembier}
K. Coulembier and R. B. Zhang
 \textit{Invariant integration on orthosymplectic and unitary supergroups},
J. Phys. A: Math. Theor. {\bf 45} no. 9  (2012), 095204.

\bibitem{de} P. Deligne, P. Etingof, D. Freed, L. Jeffrey, D. Kazhdan,
J. Morgan, D. Morrison, E. Witten (eds.),
{\it Quantum fields and strings, a course for mathematicians,} Vol. 1-2.
Material from the Special Year on Quantum Field Theory held at the Institute
for Advanced Studies 1996-1997,
Amer. Math. Soc., Providence 1999.


\bibitem{fg}
R. Fioresi and F. Gavarini,
{\it Real forms of complex Lie superalgebras and supergroups},
Comm. Math. Phys.  {\bf 397}  (2023), 937--965.

\bibitem{ga} F. Gavarini,
{\it A new equivalence between super Harish-Chandra pairs and Lie supergroups},
Pacific J. Math. {\bf 306} no. 2 (2020), 451-485.

\bibitem{gz}
I. M. Gelfand and A. Zelevinski,
 \textit{Models of representations of classical groups and their
hidden symmetries},
Funct. Anal. Appl. {\bf 18} (1984), 183-198.

\bibitem{gs2}
V. Guillemin and S. Sternberg,
 \textit{Geometric quantization and multiplicities of group representations},
Invent. Math. {\bf 67} (1982), 515-538.

\bibitem{gs}
V. Guillemin and S. Sternberg,
{\it Symplectic techniques in physics,}
Cambridge Univ. Press, 1984.


\bibitem{hc-v} Harish-Chandra, \textit{Representations of semi-simple Lie
groups V}. Amer. J. Math. {\bf 78} (1956), 1-41.

\bibitem{hc-vi} Harish-Chandra, \textit{Representations of semi-simple Lie
groups VI}. Amer. J. Math. {\bf 78} (1956), 564-628.


\bibitem{jakobsen}
 H. P. Jakobsen, {\it
The Full Set of Unitarizable Highest Weight Modules of
Basic Classical Lie Superalgebras}, Memoirs Amer. Math. Soc.
{\bf 532} (1994).

\bibitem{ka}
V. G. Kac,
 \textit{Lie superalgebras}, Adv. Math. {\bf 26} (1977), 8-96.



\bibitem{kessler}
E. Kessler,
{\it Supergeometry, Super Riemann Surfaces and
the Superconformal Action Functional}, Springer LNM, 2019.

\bibitem{kn}
A. W. Knapp,
{\it Lie Groups beyond an Introduction, 2nd. ed.},
Progr. Math. {\bf 140}, Birkh\"{a}user, Boston 2002.

\bibitem{ko}
B. Kostant,
 \textit{Quantization and unitary representations},
Lecture Notes in Math. {\bf 170}, pp.87-208,
Springer-Verlag, New York/Berlin 1970.

\bibitem{ma} Y.~I.~Manin, \textit{Gauge field theory and complex geometry, 2nd. edition};
Springer-Verlag, Berlin-New York, 1997.

\bibitem{mw}
J. Marsden and A. Weinstein,
 \textit{Reduction of symplectic manifolds with symmetry},
Rep. Math. Phys. {\bf 5} (1974), 121-130.

\bibitem{ns}
K.-H. Neeb and H. Salmasian,
{\it Lie supergroups, unitary representations, and invariant cones},
in ``Supersymmetry in mathematics and physics'',
Lecture Notes in Math. {\bf 2027}, pp. 195-239,
Springer, Heidelberg 2011.



\bibitem{pa} M. Parker,
{\it Classification of real simple Lie superalgebras of classical type,}
J. Math. Phys. {\bf 21} (1980), 689-697.





\bibitem{sj} R. Sjamaar,
{\it Symplectic reduction and Riemann-Roch formulas for multiplicities,}
Bull. Amer. Math. Soc. {\bf 33} (1996), 327-338.


\bibitem{vsv2} V.~S.~Varadarajan,  {\it Supersymmetry for
    mathematicians: an introduction},  Courant Lecture Notes  {\bf 1},
  AMS, 2004.



\end{thebibliography}
\end{document}